\documentclass[]{siamltex}
\usepackage{amssymb,amsmath}
\usepackage{color}
\usepackage{graphics}
\usepackage{graphicx}
\usepackage{latexsym}
\usepackage{arydshln}

\def\rr{\mathbb R}

\def\AA{\mathcal A}

\def\rank{{\rm rank}}
\def\Div{{\rm div}}

\def\Ker{{\rm Ker}}

\def\div{{\rm div\,}}

\usepackage{color}

\newtheorem{theoreme}{Theorem}[section]                   
\newtheorem{lem}{Lemma}[section]
\newtheorem{example}{Example}[section]
\newtheorem{remark}{Remark}[section]

\usepackage[colorlinks=true]{hyperref}
\hypersetup{urlcolor=blue, citecolor=red}
\newcommand{\Red}{\color{black}}

\title{\bf Polynomial Stability  for Weakly Coupled System with Partial Controls	\thanks{This work is supported by the National Natural Science Foundation of China (grants No. 12271035, 12131008) and Hebei Natural Science Foundation (grant No. A2025105012).}}

\author{
Bopeng Rao\thanks{Institut de Recherche Math\'{e}matique Avanc\'{e}e, Universit\'{e} de Strasbourg, Strasbourg 67084, France  (email:
bopeng.rao@math.unistra.fr).}
		  \and
		Qiong Zhang\thanks{Corresponding Author. School of Mathematics and Statistics, Beijing Key Laboratory on MCAACI,
			Beijing Institute of Technology, Beijing 100081, China (email: zhangqiong@bit.edu.cn)}
	}

\begin{document}
		
\maketitle

\begin{abstract}
We study the stability of general weakly coupled systems subject to a reduced number of local or boundary controls. We show that, under Kalman's rank condition, the exponential stability of the underlying scalar equation implies polynomial stability of the full coupled system. Moreover, the decay rate remains unchanged regardless of the number of equations in the system. The proof relies on resolvent estimates and a clever exploitation of Kalman's rank condition to ensure effective transmission of damping across the coupled equations.
The abstract result is applied to several concrete models, including systems of wave equations with local viscous, local viscoelastic, or boundary damping; systems of plate equations with internal damping; and thermoelastic systems of  type III. Moreover, the optimality of the decay rate  is established via spectral analysis.
\end{abstract}

{\small {\it Keywords. coupled partial differential equations, control, stability,
Kalman's  rank conditions}

\smallskip


\bigskip

\section{Introduction and the main result}

Let {\Red $ H_1\subset H_0$} be two separable Hilbert spaces equipped with the inner products
 $\langle\!\langle \cdot,  \cdot\rangle\!\rangle_{H_0}$ and $\langle\!\langle \cdot,  \cdot\rangle\!\rangle_{H_1}$,  respectively. We assume that the embeddings $  {\Red H_1\subset H_0}$ are dense and compact.
Taking ${H_0}$ as the pivot space, we denote by $H_{-1}$ the dual of  $H_1$.
In particular, the canonical duality pairing between $H_{-1}$  and  $H_1$  satisfies
$$\langle\!\langle\phi, \psi\rangle\!\rangle_{H_{-1}; H_1}=\langle\!\langle \phi, \psi\rangle\!\rangle _{H_0}, \quad  \forall \psi\in  H_{1}, \; \phi\in {H_0}. $$

Define $L\,: \, H_1 \to  H_{-1}$  as the duality mapping
\begin{equation}\langle\!\langle L\phi, \psi\rangle\!\rangle_{H_{-1}; H_1}
=\langle\!\langle \phi, \psi\rangle\!\rangle _{H_1},\quad \forall\;  \phi, \; \psi\in H_1.\label{1.1}\end{equation}
Let $g\,: \,H_1 \to Z\subseteq H_{-1}$ be a bounded linear operator  into another  Hilbert space $Z$,   with adjoint operator $g^*$ defined by
\begin{align}\label{1.11}
 \langle\!\langle g^*g\phi, \psi\rangle\!\rangle_{H_{-1}; H_1}=
 \langle\!\langle   g\phi, g\psi\rangle\!\rangle_{Z}, \quad \forall\; \phi, \; \psi \in H_1.
 \end{align}
We consider the scalar evolution equation with  $\lambda \geqslant 0$:
\begin{equation}\label{1.6}
\left\{\begin{array}{l}
u''+Lu +\lambda u +g^*gu' =0, \;\;t>0
\\
 \noalign{\medskip}
u(0) =u_0,\;\; u'(0) = u_1, \end{array}\right.\end{equation}
where $ ``\  '  \ "$ denotes the time derivative. Assume that the initial data $(u_0,\;u_1) \in H_1 \times H_0$.
Since the damping operator $g^*g$ is non-negative, this abstract formulation encompasses a wide variety of partial differential equations, including models with global, local, or boundary damping, as well as coupled systems such as thermoelastic models. Well-posedness of such systems follows from standard arguments (see \cite{kavian2022}).

Moreover, we assume that the system \eqref{1.6} is exponentially stable, i.e., there exists positive constants $M,\; \omega$ such that  the solution satisfies
{\Red $$
\|(u(t),\, u'(t))\|_{H_1\times H_0} \leqslant Me^{-\omega t}\|(u_0,\, u_1)\|_{H_1\times H_0}, \;\; t >0.
$$}
We analyze the stability of  weakly coupled  systems consisting of $N$ scalar equations of the form described above, where the coupling occurs through the displacement terms and controls are applied to only a subset of the state components:

The stabilization and control of   coupled systems  is a fundamental problem in control theory.
 In finite-dimensional systems, the classical Kalman's rank condition ensures exponential stability with arbitrary decay rates. For infinite-dimensional systems, however, the situation is more subtle.
 This paper introduces a systematic methodology that combines   known exponential stability   for scalar PDE with algebraic arguments to achieve polynomial stability for weakly coupled systems with a reduced number of controls.  
 Our framework  provides a simple criterion to verify polynomial stability in a broad class of coupled PDEs, including coupled hyperbolic PDEs with local  or  boundary controls, coupled plate equations with local  or boundary   controls, and coupled thermo-elastic systems.
 In the most cases, the obtained  polynomial  decay rates are sharp.

This problem has a long history.
Russell \cite{russell1993}, inspired by thermo-elastic systems, introduced an abstract framework for indirectly damped systems, in which the damping arises from the thermal conduction, that   may follow  Fourier's law,  Cattaneo's law or  the Gurtin-Pipkin's law.  His work showed that damping applied to one component of the state can be transmitted to others through coupling, thereby stabilizing the entire system. Since then, many works have sought conditions under which damping and coupling operators guarantee  system  stabilization, yielding results ranging from exponential to polynomial decay (see, e.g., \cite{delloro2013, han2023, hao2015}).

A related direction  concerns  the stabilization of coupled systems using a reduced number of controls, where the key goal is to determine the minimal set of equations requiring actuation. For instance, coupled wave-wave models with control applied to a single equation  have been studied in \cite{liurao2007}, where weak coupling leads to polynomial decay. In contrast, strong coupling yields exponential stability, as shown in \cite{alabauwang2017}.
In \cite{alabau2002,   guglielmi2017}, the authors introduced an abstract framework for weakly coupled systems of two second-order evolution equations and analyzed their  polynomial stability. Their results apply to   classical  coupled models-such as wave-wave, wave-Petrovsky, and Petrovsky-Petrovsky systems with single internal or boundary damping.
Further contributions on controllability and synchronization can be found in \cite{ duprez2016, fernandez2010, liard2017, mauffrey2013, steeves2019, book} and references therein, while more recent works extend these results to nonlinear, degenerate, or star-shaped systems \cite{hassi2020, lefter2022, raoliuzhang}.

For PDE models such as wave and plate equations, stabilization via local or boundary damping has been well studied under the geometric control condition  (see, e.g.,\cite{bardos1992, komornik, lagnese1989}).
Furthermore, for abstract systems governed by a single equation with global damping, stability and regularity properties have been thoroughly investigated (see, e.g., \cite{chen1989,chen1990,delloro2021,liu2015}).
However, for  abstract systems  with local damping, the corresponding theory remains less developed. Existing approaches (e.g., \cite{chen1991, chen1989, kavian2022}) establish exponential or polynomial stability by carefully analyzing the damping operator on eigenspaces.

In this work, we introduce  a general framework to study weakly coupled  PDE systems with partial local or boundary damping. Our aim is to  characterize  how the algebraic structures of the coupling   and control terms  influence system stabilization and decay rates.  By introducing a Kalman's rank condition to ensure sufficient coupling strength, and assuming the control can exponentially stabilizing the corresponding scalar problem, we  establish polynomial stability for the full coupled system.
 Our approach combines PDE stability theory with algebraic arguments, providing a simple and effective criterion.
  We apply our main results on coupled wave equations with local or boundary damping, coupled plate equations with internal damping, and coupled thermoelastic systems. For specific cases,   the optimality of the decay rates is  confirmed using spectral analysis.

The remainder of the paper is organized as follows. Section 2 states the main results and recalls basic properties of the Kalman's rank condition. Section 3 contains the proofs. Section 4 applies the abstract results to several PDE models, including coupled wave, plate, and thermoelastic systems. Section 5 establishes the sharpness of the results and discusses directions for future research.

\bigskip

\section{Polynomial stability and preliminaries}
In what follows, we consider coupled  systems composed of scalar equations of the form \eqref{1.6}.
Let $\mathcal H_0,\; \mathcal Z,\; \mathcal H_1$ denote the product spaces associated with
$H_0, \; H_1,\; Z$, respectively:
\begin{equation}
\mathcal H_0 = (H_0)^N,\quad \mathcal H_1 = (H_1)^N,\quad \mathcal Z = Z^N.
\end{equation}
 For a vector $U= (u^{(1)},  \ldots, u^{(N)})^T $, we define the vector-value operators $\mathcal L$ and  $\mathcal G$ as
\begin{equation} \label{z1.6}
\mathcal L U = \begin{pmatrix}Lu^{(1)}\\  \vdots\\ Lu^{(N)}\end{pmatrix}, \quad \mathcal GU = \begin{pmatrix} gu^{(1)}\\  \vdots\\ g u^{(N)}\end{pmatrix}.
\end{equation}
 We   consider the following second-order evolution system:
\begin{equation}
\left\{
\begin{array}{l}
U'' + \mathcal L U +AU  + D\mathcal G^*\mathcal GU' =  0, \;\; t>0,
\\  \noalign{\medskip}
U(0) = U_0, \;\; U'(0) =U_1,
\end{array}
\right.
\label{1.3}
\end{equation}
where  the initial data $(U_0, U_1) \in \mathcal H_1 \times \mathcal H_0 $, and the matrices
 $A=(a_{ij})_N$ and $D=(d_{ij})_N$ are symmetric positive semi-definite, representing the coupling matrix and the control matrix, respectively.

Define the linear operator $\mathcal A$ by
\begin{equation}\label{3.1}\mathcal A(U,  V)= (V, -\mathcal  LU -AU -  D\mathcal G^*\mathcal GV)\end{equation}
with  domain
\begin{equation}
\frak D(\mathcal A) = \big\{(U, V)\in \mathcal H_1\times  \mathcal H_1: \quad
\mathcal  LU + AU+ D\mathcal G^*\mathcal GV\in  \mathcal H_0\big\}.
\label{3.2}
\end{equation}
Then system \eqref{1.3} can be rewritten in the abstract form:
\begin{equation}
\left\{\begin{array}{l}(U, U')' =\mathcal A(U, U') ,\quad t>0,
\\   \noalign{\medskip}   \displaystyle
(U(0), U'(0)) = (U_0, U_1) \in \mathcal  H_1\times \mathcal H_0.
\end{array}\right.\end{equation}

Due to the presence of the coupling term
 $AU$, system \eqref{1.3} may be referred to as a weakly coupled system, particularly in contrast to systems with velocity or higher-order coupling. It is clear that the rank of $D$ determines the number of effective controls.
It is known that system \eqref{1.3} is not exponentially stable when $\rank (D)<N$ (see \cite{Siam2016,raoliuzhang}). This is to be expected, since in this case the number of controls is reduced relative to the number of state components.

Our goal is to analyze the polynomial stability properties of the system under certain assumptions on the matrices $A$ and $D$.
We thus introduce the following  Kalman's rank condition (\cite{book})
\begin{equation}
\rank  (D,\, AD,\, \ldots,\,  A^{N-1}D)= N.
\label{1.4}\end{equation}

 The following result gives the polynomial stability of the coupled system \eqref{1.3}.

 \medskip

\begin{theoreme}\label{th1.1}   Assume that   the following conditions hold:
\begin{enumerate}
\item[{\rm (i)}]
Let   $A$ and $D$  be  symmetric and positive semi-definite matrices. Assume that $A, D $  satisfy the Kalman's rank condition \eqref{1.4}
and  $\|AD-DA\|$ is sufficiently small;
\item[{\rm (ii)}] The scalar problem \eqref{1.6}
 is  exponentially stable in the space $H_1\times H_0$;
 \item[{\rm (iii)}] For any function $u\in H_1$,   the following regularity condition holds:
\begin{eqnarray}
\|gu\|_{Z} \leqslant c\|u\|^{r}_{H_1}\|u\|^{1-r}_{H_0}, \quad 0\leqslant  r\leqslant 1.
\label  {1.5}
\end{eqnarray}
\end{enumerate}
  \noindent
 Then there exists a constant $M>0$, such that   for any  initial data $(U_0, U_1) \in \frak D(\mathcal A)$, the solution $U$ to system  \eqref{1.3}  satisfies the following polynomial decay estimate:
\begin{equation}
\|(U, U')\|_{\mathcal  H_1\times \mathcal H_0} \leqslant \frac{M}{t^{1/2(1+r)}}\|(U_0, U_1)\|_{\frak D(\mathcal A)},\;\; t\geqslant 1.
\label  {1.7}
\end{equation}
\end{theoreme}

\begin{remark} Several remarks are in order:
\begin{enumerate}

\item[{\rm (i)}]
Theorem \ref{th1.1} can be interpreted as follows:  under suitable assumptions, the uniform stability of the scalar equation \eqref{1.6} implies the polynomial stability of the coupled system \eqref{1.3}.
This result provides a simple and effective approach to a problem that is otherwise difficult-namely, establishing the polynomial stability of a system consisting of $N$ coupled partial differential equations.

\item[{\rm (ii)}]
An important feature of the Kalman's rank condition \eqref{1.4} is that the rank of $D$ can be very small (see Lemma \ref{th2.2}).
Even if only a small subset of equations is directly damped, the coupled system \eqref{1.3} still enjoys a polynomial decay of the total energy.
Moreover, this decay rate is independent of the number $N$ of equations, depending instead on the regularity properties of the damping operator.
This is a typical characteristic of polynomial stability (see \cite{Hao, Rao}).

{ \Red \item[{\rm (iii)}]
As demonstrated in \cite[Theorem 8.11]{book}, merely satisfying Kalman rank condition  \eqref{1.4} is insufficient to guarantee the uniqueness of \eqref{3.7}.
If $A$ and $D$ commute,  Kalman rank condition yields $\rank(D) = N$, and consequently, by Holmgren's uniqueness theorem, the uniqueness of \eqref{3.7} follows. However, this is of little significance, as we desire $\rank(D)$ to take a small value. Instead of $AD - DA=0$, we require that $\|AD - DA\|$ is sufficiently small. In this case, there exist matrices $\widehat A$ and $\widehat D$ that  commute and are respectively close to $A$ and $D$  and the matrices $A$ and $D$ are termed "almost  commutative" \cite{Rosen}.  Furthermore, the rank of $D$ in the Kalman rank condition \eqref{1.4} can be very small, which is what we hope for. Therefore, the smallness of $\|AD-DA\|$ appears to be a natural and appropriate complement to  Kalman rank condition \eqref{1.4} to   guarantee the uniqueness of \eqref{3.7}.}

\item[{\rm (iv)}]
The class of damping operators $g \in C(H_1, Z)$ considered here covers a broad range of mechanisms, including viscous damping, Kelvin-Voigt damping, boundary damping, and type~III thermoelastic damping.

\item[{\rm (v)}]
As mentioned in Section~1, since the introduction of the concept of indirect stabilization by Russell \cite{russell1993}, a substantial body of work has been devoted to the polynomial stability of indirectly damped systems.
However, relatively few studies have addressed such problems for systems consisting of $N$ PDEs with local or boundary damping.
Furthermore, our framework can also be applied to analyze  the long time behavior of the coupled system of $N$ thermoelastic  equations (\ref{th-3}).

\item[{\rm (vi)}] As will be shown in Section 5, the optimality of the decay rate in \eqref{1.7} can be established for several special cases, including the weakly coupled wave system with viscous, viscoelastic, or boundary damping in one-dimention.
\end{enumerate}
\end{remark}

To prove Theorem \ref{th1.1}, we shall use the following result, which  characterizes the polynomial energy decay rate of a $C_0$-semigroup of contractions in terms of the growth of the resolvent of its infinitesimal generator on the imaginary axis.

\medskip

\begin{lem}\label{th3.3}  (\cite{borichev2010, raoliu2005})
Let $\mathcal  A  $ generate a bounded $C_0$-semigroup $e^{t\mathcal A}$ on   Hilbert space  ${\mathcal H } $.
Assume that
\begin{equation}
    \label{3.13}
 i\,\mathbb{R}\subset \rho(\mathcal A).
   \end{equation}
Then the semigroup $e^{t\mathcal A}$ decays polynomially at a rate of order $\displaystyle  \theta$, i.e.,
 there exists a constant $M>0$, independent of $U$, such that
 \begin{equation}
 \|e^{\mathcal  A t}U\|_{H} \leqslant\frac {M}{t^{\theta}}{\Red \|U\|_{\frak D(\mathcal  A)}},\;\; \forall \;t\geqslant  1,
 \end{equation}
if and only if
   \begin{equation}
   \label{3.14}
\varlimsup\limits_{\lambda\in \mathbb{R},|\lambda|\to+ \infty}|\lambda|^{-{1/\theta}} \|(i\,\lambda I-\mathcal A)^{-1}\|< +\infty.
 \end{equation}
 \end{lem}

 The next result is frequently used to determine the optimal decay rate for a $C_0$-semigroup.

\medskip

  \begin{lem}\label{th3.4}  (\cite{RL}) Let
 $(\beta_n)_{n\geqslant 1}$ be a sequence of eigenvalues of $\mathcal A$ such that
\begin{equation}
\label{eq-opti}
Re (\beta_n)\sim - \frac{1}{|Im (\beta_n)|^{1/\theta}},\quad \theta>0.\end{equation}
Then,  the decay rate of  the semigroup $e^{t\mathcal A}$  cannot exceed
 $\mathcal O(t^{-\theta})$.
 \end{lem}

\medskip
The Kalman's rank condition \eqref{1.4} characterizes the coupling structure in system \eqref{1.3} and is of central importance  in the analysis of control and synchronization for distributed parameter systems (see, for instance, \cite{book} and the references therein).
For the convenience of the reader, we briefly review key aspects of the Kalman's rank condition below and refer to \cite{book} for a more comprehensive treatment and applications.
We first recall a fundamental property of the Kalman's rank condition.

\medskip
\begin{lem} (\cite[Lemma 2.5]{Siam2016})  \label{th2.0}
Let $A, D$ be symmetrical matrices  of order $N$.  Then  Kalman's  rank condition
  \begin{equation} \rank(D, AD,\ldots, A^{N-1}D) =N-p\label{2.0} \end{equation}
holds if and only if $p$ is the dimension of the largest    $A$-invariant subspace contained in $\Ker(D)$.
\end {lem}

{\Red Since $A$ is symmetric and positive semidefinite,   there exists an orthogonal matrix $P$ such that $PAP^{-1}$ is diagonal. Under the change of variables $\widetilde U = PU$, system \eqref{1.3} preserves its structure, and Kalman's condition \eqref{1.4} remains unchanged, since it is invariant under orthogonal transformations. Therefore, without loss of generality, we may assume that $A$ is diagonal of the form}
\begin{equation}\label{2.1} A = diag(\overbrace{\lambda_1,\ldots,
\lambda_1}^{\sigma_1}, \cdots\ldots, \overbrace{\lambda_m,\ldots,
\lambda_m}^{\sigma_m}),\end{equation}
where each
$\lambda_l\geqslant 0$ is an eigenvalue of $A$  with  multiplicity $\sigma_l \ (l= 1, \ldots, m)$.

Define
\begin{equation}\notag   \mu_0= 0,  \quad \mu_l = \mu_{l-1}+ \sigma_l, \quad l = 1,\ldots, m\end{equation}
and write the matrix $D$ as
$$D = (d_1,\ldots,  d_{\mu_1}, \ldots \ldots,
d_{\mu_{m-1}+1},\ldots,  d_{\mu_m}),$$
where $d_i\in \mathbb R^N$ denotes the $i$-th column  vector of $D$.

\medskip

The next result relates the structure of $D$ to that of $A$.
\medskip

\begin{lem}  (\cite{book}) \label{th2.1}
Kalman's rank condition \eqref{1.4}  holds   if and only if for each $1\leqslant l\leqslant m$, the set of column vectors  $\{d_{\mu_{l-1}+1},\ldots, d_{\mu_l}\}$ is linearly independent.
\end{lem}

The following result further clarifies the dependence of $\rank(D) $ on the  structure of
matrix $A$ under the Kalman's rank condition \eqref{1.4}.
In particular, if  $A$ has distinct eigenvalues,  it is possible to choose
 $D$ such that $\rank(D)=1$.
\medskip

\begin{lem}  (\cite{book})\label {th2.2}
Assume that in \eqref{2.1} the multiplicities satisfy
$$\sigma_1\geqslant\ldots   \geqslant\sigma_m>0.$$
Then there exists a symmetrical positive semi-definite {\Red matrix} $D$ such that  $\rank (D)=\sigma_1 $  and   the Kalman's rank condition \eqref{1.4} holds.
\end{lem}

\medskip

Finally, we recall a basic property related to the  Kalman's rank condition.
\medskip

\begin{lem}  (\cite{book})  \label {th2.4}
Let
\begin{align}\label{2.2}D_l  = (d_{\mu_{l-1}+1},\ldots, d_{\mu_l}), \quad D= (D_1,  \ldots, D_m),
\end{align}
and
\begin{align}
\label{2.3}U_l = \begin{pmatrix}u^{(\mu_{l-1}+1)}\\\vdots\\ u^{(\mu_l)}\end{pmatrix}, \quad U=
\begin{pmatrix}U_1\\ \vdots\\ U_m\end{pmatrix}.\end{align}
Assume that the Kalman's rank condition \eqref{1.4}  holds. Then
\begin{equation}c\|U\|^2_{\mathbb R^N}\leqslant \|DU\|^2_{\mathbb R^N}-
\sum_{k\not =l}\langle D_lU_l, D_kU_k\rangle_{\mathbb R^N}\label {2.5}
\end{equation}
where  $\langle \cdot, \cdot \rangle_{\mathbb R^N}$  denotes the standard inner product in
  $\mathbb R^N $ and $c=\min\{\sigma_l\;:\; l=1,\cdots,m\}$.
\end{lem}

\begin{proof}
From
\begin{equation}DU = \sum_{l=1}^mD_lU_l,
\label {2.4}\end{equation}
we obtain
\begin{equation}\|DU\|_{\mathbb R^N}^2  =\sum_{l=1}^m\|D_lU_l\|_{\mathbb R^N}^2 +
\sum_{k\not =l}\langle D_lU_l, D_kU_k\rangle_{\mathbb R^N}.\label{2.6}\end{equation}
By Lemma \ref{th2.1},  the vectors $d_{\mu_{l-1}+1},\ldots, d_{\mu_l}$ are linearly independent.
Hence, the  matrix $D_l^TD_l$  is symmetric   positive  definite. Thus, there exists $c>0$ such that
$$\sum_{l=1}^m\|D_lU_l\|_{\mathbb R^N}^2=\sum_{l=1}^m\langle U_l,  D_l^TD_lU_l\rangle_{\mathbb R^{\sigma_l}}
\geqslant
c\sum_{l=1}^m\|U_l\|^2_{\mathbb R^{\sigma_l}}= c \|U\|^2_{\mathbb R^N}.$$
Combining this with \eqref{2.6} yields \eqref{2.5}.
\end{proof}

\bigskip

  \section{Proof of the main result} \label{3}
   In this section, we prove the polynomial stability of system \eqref{1.3}.
 We begin by defining  the energy space  $\mathcal H_1 \times  \mathcal H_0 $ equipped with the inner product
\begin{equation}
\label{norm}
\begin{array}{ll}
\langle \! \langle (U,V), (\widehat {U},\widehat {V})\rangle \! \rangle_{\mathcal H_1\times \mathcal  H_0}
= \langle \! \langle U, \widehat {U}\rangle \! \rangle_{\mathcal H_1} +  \langle \! \langle AU, \widehat {U}\rangle \! \rangle_{ \mathcal  H_0}
+ \langle \! \langle V, \widehat {V}\rangle \! \rangle_{\mathcal H_0},
\end{array}
\end{equation}
where $U =\big(u^{(1)}, \cdots, u^{(N)}\big)^T \in {\mathcal H_1}$ and $V =\big(v^{(1)}, \cdots, v^{(N)}\big)^T \in {\mathcal H_0},$  respectively,
$\widehat {U} =\big(\widehat {u}^{(1)}, \cdots, \widehat {u}^{(N)}\big)^T\in {\mathcal H_1}$ and  $
\widehat {V} =\big(\widehat {v}^{(1)}, \cdots, \widehat {v}^{(N)}\big)^T\in {\mathcal H_0}$.

\medskip
We first give the well-posedness of system \eqref{1.3}.
\medskip

\begin{proposition}  \label{th3.1}The operator $\mathcal A$ generates a $C_0$ semigroup of  contractions  on  the space $\mathcal H_1 \times \mathcal  H_0$. Moreover,  $0\in \rho(\mathcal A)$, and  if $g$  is compact from $H_1$ into  $Z$, then $\mathcal A^{-1}$ is compact on $\mathcal H_1 \times \mathcal  H_0$.\end{proposition}

\begin{proof}
For any $(U,  V)
\in \frak D(\mathcal A)$, using the norm defined in \eqref{norm} we compute
\begin{equation}
\begin{array}{ll}
&Re \langle \! \langle \mathcal A(U,  V), (U, V)\rangle \! \rangle _{\mathcal H_1\times \mathcal  H_0}
\\ \noalign{\medskip}  \displaystyle
=&  Re \big(\langle \! \langle V,U \rangle \! \rangle_{\mathcal H_1}  + \langle \! \langle AV,U \rangle \! \rangle_{\mathcal H_0}  +
\langle \! \langle -\mathcal  LU -AU -  D\mathcal G^*\mathcal GV ,V \rangle \! \rangle_{\mathcal H_0}\big)
 \\ \noalign{\medskip}  \displaystyle
 =&-  \langle \! \langle D\mathcal GV, \mathcal GV \rangle \! \rangle_{\mathcal Z}\leqslant 0.\end{array}
\label{3.5b}
\end{equation}
Hence $\mathcal A$ is dissipative.

Next,  for any  $(F,G) \in \mathcal H_1\times \mathcal  H_0$,   consider the equation $\mathcal A(U,  V) =(F,G)$, which  is equivalent to
\begin{equation}
\label{w1}
\begin{array}{l}
V = F\in \mathcal H_1,
\\ \noalign{\medskip}  \displaystyle
 -\mathcal  LU -AU -  D\mathcal G^*\mathcal GV= G\in \mathcal H_0.
\end{array}
\end{equation}
Taking the inner product of the second equation in \eqref{w1} with $U$ in $\mathcal H_0$ and using the first equation, we obtain
\begin{equation}
\label{w2}
\langle \! \langle \mathcal  LU +AU,U \rangle \! \rangle_{\mathcal H_{-1};\mathcal H_{1}} =- \langle \! \langle D\mathcal G^*  \mathcal GF+ G,U \rangle \! \rangle_{\mathcal H_{-1};\mathcal H_{1}}.
\end{equation}
This admits a unique solution $U\in \mathcal H_1$.   Together with \eqref{w1}, it follows that $(U, V) \in \frak D(\mathcal A) $ and therefore $0\in \rho(\mathcal A)$.

If  $g$  is compact from $H_1$ to  $Z$,  then the mapping
$$(F, G)\rightarrow (U, V)=(( \mathcal  L+A)^{-1} (-D\mathcal G^*\mathcal GF- G), F)$$
is compact on   $ \mathcal H_1\times \mathcal  H_0$, which implies  the compactness of $\mathcal A^{-1}$. \end{proof}

\medskip

 \begin{proposition}\label{th3.2} Suppose that conditions (i) and (ii) in {\Red Theorem \ref{th1.1}} are satisfied.
Then
\begin{equation}
i\mathbb R\subset \rho(\mathcal A).
\end{equation}
\end{proposition}

 \begin{proof} (i)  Let $(U, V)$ be an eigenvector of $\mathcal A$ associated with a pure imaginary eigenvalue $i\beta$, i.e.,
\begin{equation}\begin{cases}i\beta U=V &\hbox{in } \mathcal H_1,\\
i\beta V= -\mathcal L U -AU - D\mathcal G^*\mathcal GV&\hbox{in } \mathcal H_0.
\end{cases} \label{3.6}
\end{equation}
Eliminating   $V$ in (\ref{3.6}) gives
\begin{equation*}
\beta^2 U - \mathcal L U - AU - i\beta D\mathcal G^*\mathcal G U = 0.
\end{equation*}
Since $\mathcal L$ is symmetric and  $\beta \not =0$, we deduce
\begin{equation}
\beta^2 U - \mathcal L U - AU = 0\quad \hbox{and}\quad D\mathcal G^*\mathcal G U =0.\label{3.7}
\end{equation}

Noting the diagonal form of $A$ in \eqref{2.1}, we obtain that each component $u^{(i)}$ satisfies
\begin{equation}
\beta^2 u^{(i)} - L u^{(i)} - \lambda_ku^{(i)} = 0,\quad  i=\mu_{k-1}+1,\ldots, \mu_k, \quad k=1,\ldots, m.\label{3.10}
\end{equation}
For $ i= \mu_{k-1}+1,\ldots,  \mu_k, \ j=\mu_{l-1}+1,\ldots,  \mu_l,\ k\not = l,$ taking the  inner product of $i$-th equation with  $u^{(j)}$,    $j$-th equation with  $u^{(i)}$, and then subtracting shows
\begin{equation}(\lambda_k-\lambda_l)\langle \!\langle u^{(i)}, u^{(j)}\rangle\!\rangle _{\mathcal H_0} = 0.\end{equation}
Let $D_l$ and $U_l$ be defined by \eqref{2.2} and \eqref{2.3}. Then
\begin{equation}\langle \!\langle D_kU_k,  D_lU_l\rangle\!\rangle _{\mathcal H_0} = 0\end{equation}
By  Lemma \ref{th2.4},
\begin{equation}\|U\|^2_{\mathcal  H_0} \leqslant c  \|DU\|^2_{\mathcal H_0}.\label{3.111}
\end{equation}

Now multiplying \eqref{3.6} by $D$ yields
\begin{equation}\begin{cases}i\beta DU-DV =0 &\hbox{in } \mathcal H_1,
\\ \noalign{\medskip}   \displaystyle
i\beta DV+ \mathcal L DU+ADU +  {\Red \mathcal G^*\mathcal G DV} = (AD-DA)V &\hbox{in } \mathcal H_0.
\end{cases} \label{3-121}
\end{equation}
{\Red Let us write
\begin{equation}\label{3.5c}DU= (\widehat u_1,\cdots, \widehat u_N),\quad DV=  (\widehat v_1,\cdots, \widehat v_N),\quad
(AD-DA)V =( \widehat f_{1},\cdots, \widehat f_{N}).\end{equation}
Noting the diagonal form of $A$ in \eqref{2.1}, we split the abstract system  \eqref{3-121} by component   for $1\leqslant i \leqslant N$: \begin{equation}\begin{cases}i\beta \widehat u_i- \widehat v_i =0 &\hbox{in } H_1,
\\   \displaystyle
i\beta \widehat v_i+ L \widehat u_i+\lambda_i\widehat u_i +  g^*g \widehat v_i =\widehat f_{i}&\hbox{in }  H_0.
\end{cases} \label{3-121b}
\end{equation}
Then applying  the estimate \eqref {3.5}, we deduce
\begin{equation}
\|\widehat u_i\|_{H_1}^2 + \|\widehat v_i\|_{H_0}^2 \leqslant c \|\widehat f_{i}\|_{H_0}^2. \label{3.5bb}
 \end{equation}
Noting \eqref{3.5c} and taking the sum  of \eqref{3.5bb}  for $1\leqslant i \leqslant N$, we get
\begin{equation}
\|DU\|^2_{\mathcal  H_1} +\|DV\|^2_{\mathcal H_0} \leqslant c' \|AD-DA\|^2 \|V\|^2_{\mathcal  H_0}.
\end{equation}
With \eqref{3.111} remaining valid for $V$, this gives
\begin{equation}\|V\|^2_{\mathcal  H_0}  \leqslant  c\|DA-AD\|^2 \|V\|^2_{\mathcal  H_0}.\end{equation}
When $\|AD-DA\|$ is sufficiently small, we get  $ U=V=0$. Hence $\mathcal A$ has no eigenvalues on $i\mathbb R$.}

(ii) Let $(F, G)\in \mathcal H_1\times \mathcal H_0$. We look for   $(U, V)\in \mathfrak{D}(\mathcal A)$ solving
\begin{equation}
(i\beta I - \mathcal A )(U, V)= (F, G),\label {3.25b}\end{equation}
i.e.,
\begin{equation} \label {3.26b} i\beta U- V=F, \quad i\beta V + \mathcal L  U +AU + D  \mathcal G^* \mathcal GV =G.\end{equation} Eliminating $V$ yields
\begin{align} V&=i\beta U-F,\label{3.27b}\\
-\beta^2U + \mathcal L U +AU +  i\beta  D \mathcal G^*\mathcal GU &=G+i\beta F +  D  \mathcal G^*\mathcal G G.\label{3.28b}\end{align}

\noindent
Define the linear operator $\mathcal A_0 : \, \mathcal H_1 \to \mathcal H_{-1}$   by
\begin{equation}\mathcal A_0U= \mathcal L U +AU +  i\beta  D \mathcal G^* \mathcal GU.\end{equation}
By the Lax-Milgram's Lemma \cite{brezis}, $\mathcal A_0$ is a continuous isomorphism from $\mathcal  H_1$ onto $\mathcal H_{-1}$.  Equation \eqref{3.28b} rewrites as \begin{equation}
(I-\beta^2\mathcal A_0^{-1}) U  = \mathcal A_0^{-1}(G+i\beta F+  D \mathcal G^* \mathcal G G). \label{3.29b}
\end{equation}

\noindent
If $U\in\Ker(I-\beta^2\mathcal A_0^{-1})$, then
\begin{equation}\beta^2 U=\mathcal L U +AU +  i\beta  D \mathcal G^* \mathcal GU.\end{equation}
Setting $V=i\beta U$ gives again system \eqref{3.6}, whose uniqueness implies
$U=0$. Thus $\Ker(I-\beta^2 \mathcal A^{-1}_0)=\{0\}.$    Furthermore,  noting that   $\mathcal A_0$ is a continuous isomorphism from $\mathcal  H_1$ onto $\mathcal H_{-1}$. Due to the compact embedding \eqref{3.1}, $\mathcal A_0^{-1}$ is compact  from  $\mathcal  H_{-1}$ onto $\mathcal H_1$. Fredholm's alternative (\cite{brezis}) implies that \eqref{3.29b} has a unique solution   $U\in H_0$.
Setting $V$ by \eqref{3.27b}, we obtain a unique $(U,V)$ solving \eqref{3.25b}.  The proof is complete. \end{proof}

\medskip
\medskip

\begin{proposition} Assume that   the corresponding scalar problem \eqref{1.6}
 is   exponentially stable in $H_1\times H_0$.  Then  there exists a constant $c>0$, such that  for all   $\beta\in \mathbb R$ and
 $(f_1, f_0)\in
 H_1\
\times H_0$,  the solution $(u, v)$  to  the  system
\begin{equation}\begin{cases}i\beta u-v=f_1 &\mbox{in }\;\;H_1,\\
i\beta v+ L u+\lambda u +g^*gv= f_0   &\mbox{in } \;\; H_0
\end{cases} \label{3.4}
\end{equation} satisfies  the estimate
\begin{equation}
\|u\|_{H_1}^2 + \|v\|_{H_0}^2 \leqslant c(\|f_1\|_{H_1}^2 + \|f_0\|_{H_0}^2). \label{3.5}
 \end{equation}
 \end{proposition}
 \begin{proof}By the classic theory on semi-groups (see \cite{huang, pruss}), the resolvent  of the infinitesimal generator
 \begin{equation} \widetilde{\mathcal A}(u,
v) = (v,
-Lu -\lambda  u -g^*gv)\end{equation}
is uniformly bounded on the imaginary axis:
$${\Red \sup_{\beta\in \mathbb R}}\|(i\beta I-\widetilde{\mathcal A})^{-1}\|\leqslant c<+\infty.$$
Writing $$ (i\beta I-\widetilde{\mathcal A})(u, v)= (f_1, f_0), $$
we immediately obtain the estimate \eqref{3.5}.
\end{proof}

\medskip

\begin{remark}\label{m-p}
Before giving the proof of the main theorem \ref{th1.1}, we recall some algebraic preliminaries.
Let $p_1, p_2,\ldots, p_N$ be an orthonormal basis of $\mathbb R^N$. Let
$$0<\delta_1\leqslant \ldots \leqslant \delta_d$$such that
$$Dp_j= \delta_jp_j, \quad  j=1,\ldots, d  \quad \hbox{and}\quad
Dp_j= 0,\quad \quad j=d+1,\ldots, N.$$

Define
$$M = \begin{pmatrix} \widehat M&0\\
0&0\end{pmatrix}_{N\times N} \quad \hbox{and}\quad
\widehat M =  \begin{pmatrix}  \delta_1\\
&\ddots\\ &&\delta_d \end{pmatrix}_{d\times d}.$$
Setting  $P= (p_1, p_2,\ldots, p_N)$ such that $D=P M P^T$, we write
$$P=(\widehat P, \; \widetilde P) \quad \hbox{with } \;\widehat P =(p_1,\ldots, p_d) \quad \hbox{and}\quad
 \widetilde P =(p_{d+1},\ldots, d_N).$$
For $U= (u^{(1)},  \ldots, u^{(N)})^T $,  we have
$$
P^TU= \begin{pmatrix}  \widehat P^T U\\ \widetilde P^TU \end{pmatrix} = \begin{pmatrix}  \widehat U \\ \widetilde U \end{pmatrix}.
$$
Then,
$$DU = P MP^TU =  P  \begin{pmatrix} \widehat M&0\\
0&0 \end{pmatrix} \begin{pmatrix}\widehat U\\ \widetilde U \end{pmatrix}= P \begin{pmatrix}\widehat M\widehat U\\ 0 \end{pmatrix},$$
which implies
\begin{equation}\|DU\|_{\mathbb R^N}^2 = \|\widehat M\widehat U\|_{\mathbb R^d}^2\leqslant \delta_d^2\|\widehat U\|_{\mathbb R^d}^2. \label  {3.15} \end{equation}
\end{remark}
\medskip

We now turn to the proof of Theorem \ref{th1.1}, which relies on the characterization in Lemma \ref{th3.3}.

\medskip

{\it Proof of Theorem \ref{th1.1}.} The condition  \eqref{3.13}  follows directly  from  Proposition  \ref{th3.2}.
Suppose, by contradiction, that \eqref{3.14} fails with  $\theta=1/2(1+r)$. By the uniform boundedness theorem,  there exists a sequence  of $\{\beta_n\} \in \rr$ with $\beta_n\to+\infty$ (assume $\beta_n>0$ without loss of generality) and a sequence of  elements
$(U_n, V_n) \in \frak D(\AA)$  such that
\begin{equation}\|(U_n, V_n)\|_{\mathcal H_1\times \mathcal  H_0}  = 1,\quad n\geqslant 1\label  {3.17}
\end{equation}
and
\begin{equation}
\lim_{\beta_n\rightarrow +\infty}\beta_n^{2(1+r)}\| (i\beta_n I - \AA)
(U_n, V_n)\|_{{\mathcal H_1\times \mathcal  H_0} } = 0.\label  {3.18}
\end{equation}

We shall show that $\|(U_n, V_n)\|_ {\mathcal H_1\times \mathcal  H_0} =o(1)$, which  contradicts \eqref  {3.17}.  For clarity, the proof  is  divided into several steps.

\medskip

(i)  Since  $W_n=(U_n, V_n)$ is bounded in ${\mathcal H_1\times \mathcal  H_0} $,   it follows from \eqref  {3.18}   that
\begin{align*}&\beta_n^{2(1+r)}\langle \! \langle {\Red i\beta_n} W_n - \AA W_n, W_n\rangle \!\rangle_{\mathcal H_1\times \mathcal  H_0}\\
=\;&\beta_n^{2(1+r)}\big(i\beta_n \|W_n \|_{\mathcal H_1\times \mathcal  H_0} ^2 - \langle\!\langle \AA W_n,  W_n\rangle\!\rangle _{\mathcal H_1\times \mathcal  H_0} \big)=o(1).
\end{align*}
Using  \eqref{3.5b}, we deduce
\begin{equation}\langle \! \langle D\mathcal GV_n, \mathcal GV_n\rangle\!\rangle_{\mathcal  Z}= \frac{o(1)}{ \beta_n^{2(1+r)}}, \label  {3.19} \end{equation}
namely,
\begin{equation}\label  {3.22} \|D\mathcal GV_n\|_{\mathcal Z} = \frac{o(1)}{ \beta_n^{1+r}}.
\end{equation}

\medskip
(ii) We rewrite \eqref  {3.18} as
\begin{equation}\begin{cases}\displaystyle\beta_n^{2(1+r)}(i\beta_n U_n-V_n)=F_n\rightarrow 0  &\hbox{in } \mathcal H_1,\\  \noalign{\medskip}
\beta_n^{2(1+r)}(i\beta_n V_n+ \mathcal L U_n+AU_n +D\mathcal G^*\mathcal GV_n) = G_n \rightarrow 0  &\hbox{in } \mathcal H_0.
\end{cases} \label  {3.23}
\end{equation}
From the first equation in \eqref{3.23}
\begin{equation}V_n=i\beta_n U_n - \frac{F_n}{\beta_n^{2(1+r)}} \quad \hbox{in } \mathcal H_1.\label  {3.24}
\end{equation}
Combining this with  \eqref{3.17}  yields
\begin{equation}
\|U_n\|_{\mathcal H_1}=\mathcal O(1),\quad \|U_n\|_{\mathcal H_0} =\frac{\mathcal O(1)}{\beta_n}.
\label  {3.25}
\end{equation}
Therefore, due to   the regularity  assumption  \eqref{1.5}  together with \eqref{3.25}, we conclude
\begin{equation}
\|\mathcal GU_n\|_{\mathcal Z} \leqslant c\|U_n\|^{r}_{\mathcal H_1}\|U_n\|^{1-r}_{\mathcal H_0} = \frac{\mathcal O(1)}{\beta_n^{1-r}}, \quad 0\leqslant  r\leqslant 1.
\label  {3.26}
\end{equation}
Eliminating   $V_n$ in \eqref{3.23} gives
\begin{equation}
\beta_n^2 U_n - \mathcal L U_n - AU_n - D\mathcal G^* \mathcal GV_n = -\frac{G_n}{\beta_n^{2(1+r)}}- \frac{iF_n}{\beta_n^{1+2r}}.
\label  {3.27}
\end{equation}
 Taking the inner product of  \eqref{3.27} with  $U_n$ in  $\mathcal H_{0}$ and using \eqref{3.25}, we obtain
\begin{equation}
\beta_n^2 \langle\!\langle U_n, U_n\rangle\!\rangle_{\mathcal H_0} -\langle\!\langle\mathcal L U_n+D\mathcal G^* \mathcal G V_n, U_n\rangle\!\rangle_{\mathcal H_0}  -
\langle\!\langle AU_n, U_n \rangle\!\rangle_{\mathcal H_0} = \frac{o(1)}{\beta_n^{2(1+r)}}.
\label  {3.28}
\end{equation}
Finally, by    \eqref{3.22}  and \eqref{3.26},
\begin{align*} &|\langle\!\langle\mathcal L U_n+ D\mathcal G ^*\mathcal G V_n, U_n\rangle\!\rangle_{\mathcal H_0}| \\
 \noalign{\medskip}
\leqslant  \;&
\langle\!\langle\mathcal L U_n, U_n\rangle\!\rangle_{\mathcal H_{-1}; \mathcal H_1} + | \langle\!\langle D\mathcal G^*\mathcal G V_n, U_n\rangle\!\rangle_{\mathcal H_{-1}; \mathcal H_1}|
\\
 \noalign{\medskip}
 =\;&\langle\!\langle U_n, U_n\rangle\!\rangle _{\mathcal H_1}  +  | \langle\!\langle D\mathcal G V_n, \mathcal  GU_n\rangle\!\rangle_{\mathcal Z}| \\
\leqslant \;&\langle\!\langle U_n, U_n\rangle\!\rangle _{\mathcal H_1}  +   \frac{o(1)}{\beta_n^{2}}.\end{align*}
Then substituting this   into \eqref{3.28} yields
\begin{equation}\|\beta_n U_n\|^2_{\mathcal H_0} -  \|U_n\|^2_{\mathcal H_1}  = \frac{o(1)}{\beta_n^{2}}. 
\label  {3.29}
\end{equation}

(iii) Recalling the diagonal form of
  $A$ in \eqref{2.1}, we rewrite \eqref{3.27} as
\begin{eqnarray}
\beta_n^2 u_n^{(i)} - L u_n^{(i)} - g^*(D\mathcal GV_n)^{(i)} - \lambda_ku_n^{(i)}  = \frac {o(1)}{\beta_n^{1+2r}} \quad   \hbox{in } H_0,
\label  {3.30}
\end{eqnarray}
where $i= \mu_{k-1}+1,\ldots, \mu_{k},\; k=1,\ldots,m$. Here,  $(D\mathcal G V_n )^{(i)} $ denotes the $i$-th composent  of $D\mathcal G V_n $,  and $o(1)$ is a quantity vanishing in $H_0$ as $n \to \infty$.

For $ i= \mu_{k-1}+1,\ldots, \mu_{k}$ and $j= \mu_{l-1}+1,\ldots, \mu_{l}$ with $
k\not = l,$ we take the inner product of the $i$-th equation in  \eqref  {3.30}  with $u_n^{(j)}$  in $H_{0}$,  and of the $j$-th equation with $ u_n^{(i)}$. Since both terms are of order $\beta_n^{-1}$ in $H_{0}$,  we obtain
\begin{align} \label  {3.31}
\beta_n^2 \langle\!\langle u_n^{(i)},  u_n^{(j)} \rangle\!\rangle_{H_0} - \lambda_k \langle\!\langle u_n^{(i)},  u_n^{(j)}\rangle\!\rangle_{H_0}
=  \langle\!\langle L u_n^{(i)}+ g^*(D\mathcal G V_n )^{(i)},  u_n^{(j)}\rangle\!\rangle
_{H_0}+ \frac{o(1)}{\beta_n^{2(1+r)}},
\end{align}
and
\begin{align}\label  {3.32}
 \beta_n^2 \langle\!\langle u_n^{(j)},  u_n^{(i)} \rangle\!\rangle _{H_0}   - \lambda_l \langle\!\langle u_n^{(j)},  u_n^{(i)}\rangle\!\rangle _{H_0}
=  \langle\!\langle L u_n^{(j)}+g^*(D\mathcal G V_n )^{(j)},  u_n^{(i)}\rangle\!\rangle_{H_0}+\frac{o(1)}{\beta_n^{2(1+r)}}.
\end{align}
It is clear that
\begin{align*}&\langle\!\langle L u_n^{(i)}+ g^*(D\mathcal G V_n)^{(i)},  u_n^{(j)}\rangle\!\rangle_{H_0} \\
=&\langle\!\langle L u_n^{(i)},  u_n^{(j)}\rangle\!\rangle_{H_{-1}; H_1}
+\langle\!\langle g^*(D\mathcal G V_n  )^{(i)},  u_n^{(j)}\rangle\!\rangle_{H_{-1}; H_1}
\\=&\langle\!\langle u_n^{(i)},  u_n^{(j)}\rangle\!\rangle_{ H_1}
+\langle\!\langle (D\mathcal GV_n )^{(i)},  gu_n^{(j)}\rangle\!\rangle_{Z}. \end{align*}
By  \eqref {3.22} and \eqref{3.26},
$$ \langle\!\langle (D\mathcal G V_n)^{(j)},  gu_n^{(i)} \rangle\!\rangle _{Z}  =  \frac{o(1)}{\beta_n^2}.$$
Substituting these into \eqref{3.31}-\eqref{3.32}, we obtain
\begin{align}
\beta_n^2 \langle\!\langle u_n^{(i)},  u_n^{(j)} \rangle\!\rangle_{H_0} - \langle\!\langle u_n^{(i)},  u_n^{(j)}\rangle\!\rangle
_{H_1}  - \lambda_k \langle\!\langle u_n^{(i)},  u_n^{(j)}\rangle\!\rangle_{H_0}  \label  {3.33}  = \frac{o(1)}{\beta_n^{2}},\\
\beta_n^2 \langle\!\langle u_n^{(j)},  u_n^{(i)} \rangle\!\rangle _{H_0} -  \langle\!\langle u_n^{(j)},  u_n^{(i)}\rangle\!\rangle_{H_1} \label  {3.34}  - \lambda_l \langle\!\langle u_n^{(j)},  u_n^{(i)}\rangle\!\rangle _{H_0} = \frac{o(1)}{\beta_n^{2}}.
\end{align}
Then,  the difference  of   \eqref  {3.33}  and  \eqref  {3.34}  gives
\begin{equation}(\lambda_k-\lambda_l)\langle \!\langle u_n^{(i)}, u_n^{(j)}\rangle\!\rangle _{H_0} = \frac{o(1)}{\beta_n^{2}},\quad k\not = l.
\label  {3.35}
\end{equation}
For  $l=1,\cdots,m$, let $D_l$ and $U_{nl}$ be defined by \eqref{2.2} and \eqref{2.3} respectively.  Hence \eqref{3.35} implies
\begin{equation*}\sum_{k\not =l}\beta_n^2\langle \!\langle D_lU_{nl}, D_kU_{nk}\rangle\!\rangle _{\mathcal H_0}=o(1).\end{equation*}
Therefore, applying {\Red  Lemma \ref{th2.4}}, we obtain
\begin{equation}\|\beta_nU_n\|^2_{\mathcal  H_0} \leqslant c  \|D\beta_nU_n\|^2_{\mathcal H_0}  +o(1).\label  {3.36} \end{equation}

 (iv)  Multiplying \eqref  {3.23}  by $p_j^T$, which is defined in Remark \ref{m-p}  and setting
 \begin{equation}
\widehat u_n^{(j)}= p_j^TU_n, \quad \widehat v_n^{(j)}= p_j^TV_n,
\;\; j= 1, \cdots,d,\end{equation}
we obtain
\begin{equation}\begin{cases} \displaystyle
i\beta_n \widehat u_n^{(j)}-  \widehat v_n^{(j)}=\frac{p_j^TF_n}{\beta_n^{2(1+r)}},\\ \displaystyle
i\beta_n \widehat v_n^{(j)}+ L \widehat u_n^{(j)}+  \delta_j \widehat u_n^{(j)}+ g^*g \delta_j\widehat v_n^{(j)}   = p_j^T(D-A)U_n  +\frac{p_j^TG_n}{\beta_n^{2(1+r)}}.
\end{cases} \label  {3.37}
\end{equation}
Applying  the estimate \eqref {3.5}, we deduce
\begin{equation*}\|\widehat u_n^{(j)}\|^2_{H_1} +  \| \widehat v_n^{(j)}\|^2_{H_0} \leqslant c \|p_j^T(D-A)U_n\|^2_{H_0}
+ o(1),\quad j=1,\ldots, d. \end{equation*}
Hence,
\begin{equation*}\|\widehat U_n\|^2_{(H_1)^d} +  \| \widehat V_n\|^2_{(H_0)^d} \leqslant c \|D-A\|^2 \|U_n\|^2_{\mathcal  H_0}
+ o(1). \end{equation*}
Using \eqref{3.15}, we further obtain
\begin{equation}\label{3.46}\frac{1}{\delta_d^2}\|DV_n\|^2_{\mathcal H_0} \leqslant c \|D-A\|^2 \|U_n\|^2_{\mathcal  H_0}
+ o(1). \end{equation}
Combining this with  \eqref{3.24} yields
\begin{equation}\label{3.47}\frac{1}{\delta_d^2}\|\beta_nDU_n\|^2_{\mathcal H_0} \leqslant c \|D-A\|^2 \|U_n\|^2_{\mathcal  H_0}
+ o(1). \end{equation}
Thus, from \eqref  {3.36} and \eqref{3.47},
\begin{equation*}\|\beta_nU_n\|^2_{\mathcal  H_0}  \leqslant  c\delta_d^2 \|D-A\|^2 \|U_n\|^2_{\mathcal  H_0}  +  o(1).\end{equation*}
 Therefore, by \eqref{3.25} and   the fact that  $\beta_n\rightarrow +\infty$,  we   conclude that
\begin{equation*}\|\beta_nU_n\|^2_{\mathcal  H_0}= o(1),\end{equation*}
which, together  with \eqref  {3.29}  implies
\begin{equation*}\|\beta_nU_n\|^2_{\mathcal  H_0} + \|U_n\|^2_{\mathcal  H_1}= o(1).\end{equation*}
This yields the desired contradiction, completing the proof.   \hfill $\square$

\medskip

\begin{remark} \label{rt}
    From \eqref{3.26} in the above proof, the polynomial estimate in Theorem \ref{th1.1} also holds  if the regularity assumption \eqref{1.5} is replaced by the following condition:
    \begin{equation}\label{rt1}
\|\mathcal GU\|_{\mathcal Z} \leqslant   \frac{c}{\beta_n^{1-r}}\big(\|(U, V)\|_{\mathcal H_1\times \mathcal  H_0}+\|(\widetilde {F}, \widetilde {G})\|_{\mathcal H_1\times \mathcal  H_0}\big), \quad 0\leqslant  r\leqslant 1,
 \end{equation}
where $\beta \in {\mathbb R}$  is sufficiently large,  and $  (U, V),\;   (\widetilde {F}, \widetilde {G}) \in \mathcal H_1\times \mathcal  H_0$ satisfy
 \begin{equation} \label{rt2}
\left\{
\begin{array}{ll}
 i\beta  U -V  =\widetilde {F} &\hbox{in }  \mathcal H_1,\\
 i\beta  V+ \mathcal L U+AU  +D\mathcal G^*\mathcal GV = \widetilde {G} &\hbox{in }  \mathcal H_0.
\end{array}\right.
\end{equation}
\end{remark}

\bigskip

 \section{Examples of application}    \label{4}

In this section, we present several examples to illustrate the abstract result. Throughout, we assume that $\Omega\subset \mathbb R^n$ is a bounded domain with  smooth boundary $\Gamma$, and that the matrices $A$ and $D$ satisfy   conditions  in Theorem \ref{th1.1}.

\medskip

 \begin{example} \label{ex41}  {\bf Coupled wave equations with local viscous damping. }
\begin{equation}\begin{cases}
U''-\Delta U + AU   + D\chi_\omega U'=  0 & \rm{in }  \;\; \rr^+\times
\Omega,\\
U=0&\rm{on } \;\;
\rr^+\times \Gamma,\\
U(0)=U_0,\quad U'(0)=U_1&  \rm{in } \;\; \Omega.
 \label   {4.1} \end{cases}
\end{equation}
 \end{example}
We set
\begin{equation}
\label{space-wave}
Z= H_0= L^2(\Omega),\quad H_1= H^1_0(\Omega),\quad   H_{-1} =   H^{-1}(\Omega).\end{equation}
Multiplying  system   \eqref   {4.1} by $\Phi \in  ( H^1_0(\Omega))^N$ and integrating by parts, we obtain the variational formulation:
\begin{equation}
\int_{\Omega}\big(\langle U'', \Phi\rangle + \langle\nabla U,\nabla \Phi\rangle+  \langle AU, \Phi\rangle\big)dx+\int_{\omega} \langle DU', \Phi\rangle dx=0,\label   {4.2}
\end{equation}here and hereafter,  the symbol $\langle \cdot, \cdot\rangle $ denotes the inner product in the Euclidean space
$\mathbb R^m$
with the dimension $m$ chosen appropriately depending on the context.

Define the linear operator $L: H^1_0(\Omega) \to H^{-1}(\Omega)$ by
\begin{equation}
\label{operator-wave}
\langle\!\langle Lu, \phi\rangle\!\rangle_{H^{-1}(\Omega); H^1_0(\Omega)}  = \int_\Omega  \langle\nabla u,  \nabla \phi \rangle dx, \quad \forall\;
u, \;  \phi \in  H^1_0(\Omega).
\end{equation}
 The damping operator $g=\chi_\omega I$ is continuous  from $H^1_0(\Omega)$ into $L^2(\Omega)$ and satisfies
\begin{align*}\langle\!\langle  g^*g  v, \phi\rangle\!\rangle_{H^{-1}(\Omega); H^1_0(\Omega)} =  \langle\!\langle  g  v,  g\phi\rangle\!\rangle_{L^2(\Omega)}  =\int_\omega v \phi dx, \quad
\forall\; v, \phi \in  H^1_0(\Omega).
 \end{align*}
 Then the  variational problem \eqref   {4.2}  can be  written in the form   \eqref{1.3}.

Let $\mathcal A$ be defined by \eqref{3.1}-\eqref{3.2} with
$ {\frak D}(\mathcal A) = (( H^2(\Omega))^N\cap (  H^1_0(\Omega) )^N) \times  (  H^1_0(\Omega))^N. $
By Proposition  \ref{th3.1},  the operator $\mathcal A$ generates a C$_0$ semigroup of  contractions  on
$ (   H^1_0(\Omega))^N \times  ( L^2(\Omega))^N$.

\medskip

 \begin{theoreme} \label   {th4.1}
   Assume that
 $\omega\subset \Omega$  is a subdomain containing a neighborhood of the boundary $\Gamma$. Assume furthermore that the matrices $A$ and $D $ satisfy Kalman's rank condition \eqref{1.4} and $\|AD-DA\|$ is sufficiently small. Then there exists a constant
  $M>0$ such that,   for any    initial data $(U_0, U_1) \in \frak D(\mathcal A) $, the solution $U$ to system  \eqref{4.1}
satisfies the polynomial decay estimate:
\begin{equation} \label   {4.3}
\|(U, U')\|_{ (   H^1_0(\Omega))^N \times  ( L^2(\Omega))^N} \leqslant \frac{M}{t^{1/2}}\|(U_0, U_1)\|_{ (   H^2(\Omega))^N \times  ( H^1(\Omega))^N}, \quad   t>1.
\end{equation}
\end{theoreme}
\begin{proof}  The damping operator $g=\chi_\omega I$   satisfies  the  regularity condition \eqref{1.5} with $r=0$:
$$\|g u\|_{L^2(\Omega)} =\|u\|_{L^2(\omega)} \leqslant \|u\|_{L^2(\Omega)}.$$
Moreover, when the subdomain  $\omega$ contains a neighborhood of  $\Gamma$,
the following problem \begin{equation}\begin{cases}
u'' -\Delta u  + \lambda  u + \chi_\omega u'=  0 & \mbox{in }  \rr^+\times
\Omega,\\
u=0&\mbox{on }
\rr^+\times \Gamma
\end{cases}
\end{equation}
is exponentially  stable  for any $\lambda\ge 0$ (see \cite{bardos1992,chen1991, tebou1998}).  Therefore, by applying Theorem \ref{th1.1} with $r=0$, we obtain the polynomial decay rate \eqref{4.3},  whose optimality will be further examined in Example \ref{example1}.
\end{proof}

\medskip

 \begin{example} \label{ex42}  {\bf Coupled plate equations with local viscous damping.}
\begin{equation}\left\lbrace
\begin{array}{ll}\label   {4.4} U''+\Delta^2 U + AU +D\chi_\omega U'=  0& \rm{in } \;\;
 \rr^+\times
  \Omega,\\
U= \partial_\nu U=0& \rm{in } \;\;
\rr^+\times \Gamma£¬\\
U(0)=U_0,\quad U'(0)=U_1&  \rm{in } \;\; \Omega.
\end{array}
\right.
\end{equation}
 \end{example}

We set
\begin{equation}
\label{space-plate}Z=H_0= L^2(\Omega),\quad H_1= H^2_0(\Omega),\quad   H_{-1} =   H^{-2}(\Omega).
\end{equation}
Multiplying  system   \eqref    {4.4}  by $\Phi \in (H^2_0 (\Omega))^N $ and integrating by parts, we obtain the following variational formulation:
\begin{equation}
\int_{\Omega}\big(\langle U'', \Phi\rangle+ \langle\Delta U,\Delta \Phi\rangle+  \langle AU, \Phi\rangle\big)dx+\int_{\omega} \langle DU', \Phi\rangle dx=0. \label    {4.5}
\end{equation}

Define the linear operator $L\,:\,H^2_0(\Omega) \to H^{-2}(\Omega)$    by
\begin{align}\label{operator-plate}
\langle\!\langle Lu, \phi\rangle\!\rangle_{H^{-2}(\Omega); H^2_0(\Omega)}  = \int_\Omega \langle\nabla u,  \nabla  \rangle\phi dx, \;\;
\forall\;  u ,\;\phi \in  H^2_0(\Omega) .
 \end{align}
 The damping operator $g=\chi_\omega I$ is continuous  from $H^2_0(\Omega)$ into $L^2(\Omega)$  and satisfies
\begin{align*}\langle\!\langle  g^*g  v, \phi\rangle\!\rangle_{H^{-2}(\Omega); H^2_0(\Omega)} =\langle\!\langle  g  v,  g\phi\rangle\!\rangle_{L^2(\Omega)}  =\int_\omega v \phi dx, \;\;
\forall\;  v, \phi \in  H^2_0(\Omega).
 \end{align*}
 Then the  variational problem \eqref    {4.5}   can be rewritten in the form \eqref{1.3}.

Let $\mathcal A$ be defined by \eqref{3.1}-\eqref{3.2} with
 ${\frak D}(\mathcal A) = ((H^4(\Omega))^N\cap  (H^2_0(\Omega))^N ) \times   (H^2_0(\Omega))^N.$
By Proposition  \ref{th3.1},  $\mathcal A$ generates a C$_0$ semigroup of  contractions  on
$(H^2_0(\Omega))^N \times (L^2(\Omega))^N$.

\medskip

 \begin{theoreme}  \label  {th4.2}    Assume that
 $\omega\subset \Omega$  is a subdomain containing a neighborhood of the boundary $\Gamma$.  Assume furthermore that the matrices $A$ and $D $ satisfy Kalman's rank condition \eqref{1.4} and $\|AD-DA\|$ is sufficiently small.
  Then, there exists a constant $M>0$ such that   for any given  initial data $(U_0, U_1) \in \frak D(\mathcal A) $, the solution $U$ to system  \eqref   {4.4} satisfies the polynomial decay estimate
\begin{equation}  \label   {4.6}
\|(U, U')\|_{ (   H^2(\Omega))^N \times  ( L^2(\Omega))^N} \leqslant \frac{M}{t^{1/2}}\|(U_0, U_1)\|_{  (   H^4(\Omega))^N \times  ( H^2(\Omega))^N},\;\; t\geqslant 1.
\end{equation}
\end{theoreme}
\begin{proof} The damping operator   $g=\chi_\omega I$   satisfies   the  regularity condition \eqref{1.5} with $r=0$:
$$\|g u\|_{L^2(\Omega)} =\|u\|_{L^2(\omega)} \leqslant \|u\|_{L^2(\Omega)}.$$
Since the subdomain $\omega$ contains a neighborhood of  $\Gamma$, the following problem
\begin{equation}\begin{cases}
u'' +\Delta^2u  + \lambda u + \chi_\omega u'=  0 & \mbox{in }  \rr^+\times
\Omega,\\
\partial_\nu u =u =0&\mbox{on }
\rr^+\times \Gamma
\end{cases}
\end{equation}
is uniformly  stable for any $\lambda\ge 0 $ (\cite{lks}).  Applying Theorem \ref  {th1.1}  with $r=0$  yields the decay estimate \eqref{4.6},
whose optimality   will be further examined in Example \ref{example1}.
\end{proof}

\medskip

 \begin{example} \label{ex43} {\bf Coupled wave equations with local Kelvin-Voigt damping.}
 \begin{equation}\label   {4.7} \begin{cases}U''-\Delta U + AU -a \Div (D\nabla aU')=  0 & \rm{in }\;\;   \rr^+\times\Omega,\\
U=0&\rm{on }\;\; \rr^+\times \Gamma,\\
U(0)=U_0,\quad U'(0)=U_1& \rm{in }\;\;  \Omega,\end{cases}\end{equation}
where  $a \in C^2(\Omega)$ is  non-negative  and its support contains a neighborhood of $\Gamma$. \end{example}

We set
\begin{equation}
\label{space-kv}
Z=H^1(\Omega), \quad H_0= L^2(\Omega),\quad H_1= H^1_0(\Omega),\quad   H_{-1} =   H^{-1}(\Omega).
\end{equation}
Multiplying  system  \eqref   {4.7}  by $\Phi \in (   H^1_0(\Omega))^N$ and integrating by parts, we obtain the following variational formulation:
\begin{equation}
\int_{\Omega}\big(\langle U'', \Phi\rangle + \langle\nabla U,\nabla \Phi\rangle+  \langle AU, \Phi\rangle\big)dx+\int_{\Omega} \langle D\nabla aU', \nabla a\Phi\rangle dx =0.
\label   {4.8}
\end{equation}

Let the linear operator $L$ be defined as \eqref{operator-wave}. The damping operator $ g  = (-\Delta)^{1/2}a$  is continuous   from $H_0^1(\Omega)$ into $ L^2(\Omega)$ and satisfies
\begin{align*}\langle\!\langle  g^*g  v, \phi\rangle\!\rangle_{H^{-1}(\Omega); H^1_0(\Omega)} =\langle\!\langle  g  v,  g\phi\rangle\!\rangle_{L^2(\Omega)}
=\int_{\Omega} \langle\nabla (av), \nabla (a\phi )\rangle dx,
\quad \forall \; v, \phi \in H_0^1(\Omega).
 \end{align*}

 The variational problem  \eqref   {4.7}  can be  written as \eqref{1.3}.
Let $\mathcal A$ be defined by \eqref{3.1}-\eqref{3.2} with
$${\frak D}(\mathcal A) = \{(U, V)\in  (H_0^1(\Omega))^N \times ( H_0^1(\Omega))^N\;  : \;
\mathcal A (U, V) \in  ( H_0^1(\Omega) )^N\times  ( L^2 (\Omega))^N \}.
$$
By Proposition  \ref{th3.1},  $\mathcal A$ generates a C$_0$ semigroup of  contractions  on  the space
$  ( H_0^1(\Omega) )^N\times  ( L^2 (\Omega))^N$.

\medskip
 \begin{theoreme}  \label   {th4.3}     Assume that  the support of the function $a$  contains a neighborhood of the boundary $\Gamma$. Assume furthermore that the matrices $A$ and $D $ satisfy Kalman's rank condition \eqref{1.4} and $\|AD-DA\|$ is sufficiently small.
 Then, there exists a constant $M>0$ such that   for any given  initial data $(U_0, U_1) \in \frak D(\mathcal A)$, the solution $U$ to  system  \eqref   {4.7}  has a polynomial decay rate:
\begin{equation}  \|(U, U')\|_{(   H_0^1(\Omega))^N \times  ( L^2(\Omega))^N} \leqslant \frac{M}{t^{1/4}}\|(U_0, U_1)\|_{ (   H^2(\Omega))^N \times  ( H_0^1(\Omega))^N}, \;\;  t\geqslant 1.
\label   {4.9}
\end{equation}
\end{theoreme}
\begin{proof}
The  operator  $ g  = (-\Delta)^{1/2}a$    satisfies   the regularity condition \eqref{1.5} with $r=1$:
\begin{equation}
    \label{gkv}
\|gu\|_{L^2(\Omega)} =  \|au\|_{H^1_0(\Omega)}   \leqslant c\|u\|_{H_0^1(\Omega)}.
\end{equation}
On the other hand, when the support of  $a$ contains a neighborhood of  $\Gamma$, the following scalar problem
\begin{equation}\begin{cases}
u'' -\Delta u  +  \lambda u + a\Div(\nabla( au'))=  0 & \mbox{in }  \rr^+\times
\Omega,\\
u=0&\mbox{on }
\rr^+\times \Gamma
\end{cases}
\end{equation}
is uniformly  exponentially stable for all $\lambda\ge 0$ (see \cite{raoliu2004}).  Applying  Theorem \ref  {th1.1}  with $r=1$, we obtain the
decay rate \eqref{4.9}, whose  optimality  will be proved  in  Exemple \ref{example2}. \end{proof}

\medskip

\begin{example} \label{ex44}   {\bf Coupled wave equations with boundary damping.}
\begin{equation}\begin{cases}
U''-\Delta U +AU=  0 & \rm{in } \;\;  \rr^+\times
\Omega,\\
\partial_\nu U + DU'=0&\rm{on } \;\;
\rr^+\times \Gamma,\\
U(0)=U_0,\quad U'(0)=U_1& \rm{in } \;\;  \Omega,
 \label    {4.10}  \end{cases}
\end{equation}
\end{example}
where $\nu$ denotes unit outer normal vector  on the boundary.

Let
\begin{equation}
\label{space-b}
Z=L^2(\Gamma), \quad H_0=\Big\{f\in L^2(\Omega)\, : \, \int_\Omega f dx = 0\Big\},
\quad H_1= H^1(\Omega) \cap H_0,\quad   H_{-1} =   (H_1)'.\end{equation}
Multiplying  system   \eqref    {4.10}  by $\Phi \in  \mathcal  H_1$ and integrating by parts, we obtain the following variational formulation:
\begin{equation}
\int_{\Omega}\big(\langle U'', \Phi\rangle  + \langle\nabla U,\nabla \Phi\rangle +  \langle AU, \Phi\rangle \big)dx+\int_{\Gamma} \langle DU', \Phi\rangle d\Gamma=0.\label    {4.11}
\end{equation}

\noindent
Let $L$ be  a linear operator   from $H^1(\Omega)$ into $H^{-1}(\Omega)$ defined  by
\begin{align*}
\langle\!\langle Lu, \phi\rangle\!\rangle_{H^{-1}(\Omega); H^1(\Omega)}  = \int_\Omega \langle\nabla u, \nabla \phi \rangle dx,
\;\; \forall \; u, \phi \in H_1.
 \end{align*}
 and let
 $g=\delta_\Gamma$ be  the trace operator  from $H^1(\Omega)$ into $ L^2(\Gamma)$ such that
\begin{align}\label   {4.12}  \langle\!\langle g^*gv, \phi\rangle\!\rangle_{H^{-1}(\Omega); H^1(\Omega)}
 = \langle\!\langle  g  v,  g\phi\rangle\!\rangle_{L^2(\Gamma)}  = \int_\Gamma v \phi\,d\Gamma,
\;\; \forall \; v, \phi\in H_1.
 \end{align}
 Then the  variational problem \eqref    {4.11}   can be  written as \eqref{1.3}.
Let $\mathcal A$ be defined by \eqref{3.1}-\eqref{3.2}.
By Proposition  \ref{th3.1},  the operator $\mathcal A$ generates a $C_0$ semi-group of  contractions  on
$ \mathcal  H_1 \times  \mathcal  H_0$.

\medskip
 \begin{theoreme} \label   {th4.4}
 Assume that     $\Omega$ satisfies the usual geometrical multiplier condition. Assume furthermore that the matrices $A$ and $D $ satisfy Kalman's rank condition \eqref{1.4} and $\|AD-DA\|$ is sufficiently small.
 Then, there exists a constant $M>0$ such that   for any given  initial data $(U_0, U_1) \in \frak D(\mathcal A) $, the solution $U$ to  system  \eqref   {4.10}  has a polynomial decay rate:
\begin{equation}  \label   {w4.13}
\|(U, U')\|_{\mathcal  H_1 \times  \mathcal  H_0} \leqslant \frac{M}{t^{1/3}}\|(U_0, U_1)\|_{\frak D(\mathcal A)}, \;\; t\ge 1.
\end{equation}
\end{theoreme}
\begin{proof}
By  Sobolev interpolation theorem,   the damping operator $g=\delta_\Gamma$  satisfies   the regularity condition \eqref{1.5}  with $r=1/2$:
\begin{equation}\label   {4.14} \|gu\|_{L^2(\Gamma)}=\|u\|_{L^2(\Gamma)} \leqslant c\|u\|^{1/2}_{H^1(\Omega)} \|u\|^{1/2}_{L^2(\Omega)}.\end{equation}
On the other hand,  when $\Omega$  satisfies  the geometrical control condition, the following scalar system  with $\lambda\geqslant 0$:
\begin{equation}\begin{cases}
u'' -\Delta u  +\lambda u=  0 & \mbox{in }  \rr^+\times
\Omega,\\
\partial_\nu u +u'=0&\mbox{on }
\rr^+\times \Gamma
\end{cases}
\end{equation}
is exponentially   stable (see \cite{komornik}).  We  can   therefore apply  Theorem \ref  {th1.1}   with $r=1/2$ to complete the proof.
  \end{proof}

\medskip

\begin{remark}
Our result improves upon the earlier work in \cite{alabau2002}, where  the decay rate of ${1\over t^{1/4}}$ was established for two weakly coupled  wave equations  with a single boundary damping.
Moreover, the following result improves further the decay rate   \eqref{w4.13} for one-dimensional problem.
\end{remark}

\medskip

\begin{example} \label{ex45}  {\bf One-dimensional coupled wave equations with boundary damping.}
We now examine in detail the one-dimensional case of \eqref{4.10} and refine the decay rate in \eqref{w4.13}, improving it from $1/t^3$ to $1/t^2$.
\begin{equation}\label   {4.15} \begin{cases}U'' - U_{xx} + AU=  0&{\rm in} \;\; \rr^+\times (0,1), \\
 U(0)=0 & {\rm in} \;\; \rr^, \\U_x(1)  +  DU'(1)=0& {\rm in} \;\; \rr^.
 \end{cases}\end{equation}
\end{example}

Let
$$
Z=\mathbb R, \quad H_0= L^2(0,1),
\quad H_1= \{f\in H^1(0,1)\,:\, f(0) =0\} ,\quad   H_{-1} =   H^{-1}(0,1).
$$
Multiplying  system   \eqref    {4.15}  by $\Phi \in  (H^1(0, 1))^N$ and integrating by parts, we obtain the following variational formulation:
\begin{equation}
\int_0^1\big(\langle U'', \Phi\rangle + \langle U_x, \Phi_x\rangle +  \langle AU, \Phi\rangle  \big)dx+\langle DU'(1), \Phi(1)\rangle =0.\label    {4.16}
\end{equation}

Define a linear operator $L: H_1 \to H_{-1}$ by
\begin{align*}
\langle\!\langle Lu, \phi\rangle\!\rangle_{H^{-1}(0, 1); H^1(0, 1)}  = \int_0^1 u_x \phi_xdx, \;\; \forall \; u,\; \phi \in H_1.
 \end{align*}
The Dirac's mass  $g=\delta_1$ is  continuous from $H^1(0, 1)$ into $ \mathbb R$,  satisfying
\begin{align} \langle\!\langle g^*gv, \phi\rangle\!\rangle_{H^{-1}(0, 1); H^1(0, 1)} = v(1)\phi(1).
 \end{align}
Hence,  the  variational problem \eqref    {4.16}   can be  written in the abstract form \eqref{1.3}.
Let $\mathcal A$ be defined by \eqref{3.1}-\eqref{3.2}.
By Proposition  \ref{th3.1},  the operator $\mathcal A$ generates a C$_0$ semigroup of  contractions  on
$ \mathcal  H_1  \times  \mathcal   H_0$.

\medskip
 \begin{theoreme} \label   {th4.5}  Assume that the matrices $A$ and $D $ satisfy Kalman's rank condition \eqref{1.4} and $\|AD-DA\|$ is sufficiently small.
 Then, there exists a constant $M>0$ such that   for any given  initial data $(U_0, U_1) \in \frak D(\mathcal A) $, the solution $U$ to system  \eqref   {4.15}  satisfies the polynomial decay estimate:
\begin{equation}  \label   {4.17}
\|(U, U')\|_{ ( H^1(0, 1))^N\times  (L^2(0, 1))^N} \leqslant \frac{M}{t^{1/2}}\|(U_0, U_1)\|_{ (H^2(0, 1))^N\times  (H^1(0, 1))^N},\;\; t\ge 1.
\end{equation}
\end{theoreme}
\begin{proof}
It is known that the following  scalar problem is exponentially stable:
\begin{equation}\begin{cases}\displaystyle
u''  -  u_{xx} +\lambda u =  0  &{\rm in} \;\; \rr^+\times (0,1),  \;\lambda\ge 0, \\
u(0) =  0, \quad u_{x}(1)  + u'(1)  =0& {\rm in} \;\; \rr^+.
\label    {4.19}  \end{cases}
\end{equation}
Note that a sharper estimate for the operator $\mathcal G$ can be obtained with $r = 0$, even though the regularity condition \eqref{1.5} with $r = 0$,
$ \|gu\|_{Z}= |u(1)| \leqslant  c\|u\|_{L^2(0, 1)}$,
does not hold for all functions $u \in H^1(0,1)$.
Nevertheless, an estimate based on hidden regularity can be established, and Remark \ref{rt} can be applied to obtain the desired result.

In fact,  multiplying two equations in \eqref{rt2}  by $xU_{x}$ and add the results up, we get
\begin{align*}
\|V(1)\|_{{\mathbb R}^N}^2  +\|U_{x}(1)\|_{{\mathbb R}^N}^2 +  \| A^{1\over2} U(1)\|_{{\mathbb R}^N}^2\\
 =   \int_0^1( \|V\|_{{\mathbb R}^N}^2 +\|U_{x}\|_{{\mathbb R}^N}^2 -\|A^{1\over2} U \|_{{\mathbb R}^N}^2 )dx
-2 \int_0^1 x(\langle \widetilde {F}_x, V\rangle_{{\mathbb R}^N}   +  \langle \widetilde {G},  U_x\rangle_{{\mathbb R}^N} )dx.
 \end{align*}
 Substituting  $i\beta  U -V  =\widetilde {F}$ into the above equality and using the assumption that $\beta$ is sufficiently large, we conclude that there exists a positive constant  $c$ such that
 \begin{equation}
 \|\mathcal GU\|_{\mathcal Z}= \|U(1)\|_{{\mathbb R}^N} \leqslant \frac{c}{\beta}\big(\|(U, V)\|_{\mathcal H_1\times \mathcal  H_0}+\|(\widetilde {F}, \widetilde {G})\|_{\mathcal H_1\times \mathcal  H_0}\big).
 \label   {4.20} \end{equation}
Then we  apply Remark~\ref{rt} with $
r=0$ to obtain \eqref{4.17},  whose  optimality will be assessed in Example \ref{example3} later.
\end{proof}

\medskip

\medskip

\begin{example}\label{ex47}   {\bf Coupled systems of thermo-elasticity.} \end{example}
The  asymptotic behavior  the system of thermo-elasticity has been intensively analyzed  in recent years. The classical  thermo-elasticity model  is based on Fourier's law, where the heat flux
is proportional to the gradient of temperature. In this case, a sudden temperature change at one point propagates instantaneously throughout the medium, which is physically unrealistic.
To address this issue, Green and Naghdi introduced the so-called thermo-elasticity model of type III \cite{gn1, gn2}, given by
\begin{equation} \label{th-1}
\begin{cases}
w'' - \mu \Delta  w- (\lambda + \mu) \nabla \div  (w) + \nabla \theta' = 0 &  {\rm in} \;\; \rr^+ \times \Omega, \\
\theta'' - \Delta \theta' - \Delta \theta + \div( w' ) = 0 & {\rm in} \;\; \rr^+ \times   \Omega, \\
w = 0, \quad \theta = 0 & {\rm on} \;\; \rr^+ \times \Gamma,
\end{cases}
\end{equation}
where  $w=(w_1, w_2, w_3)^T$ represents the displacement and  $ \theta$ denotes the temperature.

For any given test function $\phi  =(\phi_1, \phi_2, \phi_3)^T \in  (H^1_0(\Omega))^3$ and $\eta\in H^1_0(\Omega)$, we have the following variational formulation
\begin{align}\begin{cases}\displaystyle
\int_{\Omega} (\langle w'', \phi\rangle +  \mu \langle \nabla w, \nabla \phi \rangle  + \hbox{div}(w) \times  \hbox{div} (\phi ) +
 \langle \nabla \theta', \phi \rangle)dx  =0, \\\displaystyle
\int_{\Omega} ( \theta'' \eta +  \langle \nabla \theta, \nabla \eta \rangle  + \langle \nabla \theta', \nabla\eta\rangle
+    \hbox{div}(w')\times \eta ) dx  =0,\end{cases}\label{varia}
\end{align}

Let
$$u= \begin{pmatrix}   w_1\\w_2\\ w_3\\ \theta \end{pmatrix},\quad \varphi= \begin{pmatrix}  \phi_1\\\phi_2\\ \phi_3 \\ \eta \end{pmatrix},$$
and  $$
H_0 = (L^2(\Omega))^{4},\;\;  H_1 = (H^1_0(\Omega))^{4},\;\;  H_{-1} = (H^{-1}(\Omega))^{4}.
  $$
Define the linear mapping $L$ from $H_1$ onto $H_{-1} $ by
 $$  \langle \! \langle Lu, \varphi  \rangle \! \rangle_{H_1,H_{-1}} = \int_{\Omega}  (\mu \langle \nabla w, \nabla \phi \rangle  + \hbox{div} (w )\times  \hbox{div} ( \phi)
 + \langle \nabla \theta, \nabla \eta \rangle )dx,$$
 the linear  dampings $g$ and $g_0$ from $H_1$ to  $H_0$ by
  $$gu=  \begin{pmatrix}0\\(-\Delta)^{1/2}\theta\end{pmatrix} \quad \hbox{and} \quad g_0u = \begin{pmatrix} \nabla \theta\\ \hbox{div} (w)\end{pmatrix}.$$
  Then we write \eqref{varia}  into the following form
\begin{equation} \label{th-2}
u''  + Lu + g^*g u'+ g_0 u'= 0\quad \hbox{in } H_{-1},
\end{equation}
which is  exponentially stable (see \cite{zua}) when $\Omega $ satisfies the  geometry control condition.

Now let
$$ U=  \begin{pmatrix} u^{(1)}\\u^{(2)}\\\vdots\\u^{(N)}\end{pmatrix}\quad \hbox{where} \quad u^{(i)} = \begin{pmatrix}  w^{(i)}_1\\w^{(i)}_2\\w^{(i)}_3\\\theta^{(i)} \end{pmatrix},\quad i=1,\cdots, N$$
Denote by  $ \mathcal L,  \mathcal  G$  and $ \mathcal  G_0$ the vector-value operators of $L,  g$  and $ g_0$ in the space
 $$ \mathcal   H_0 = \mathcal  (L^2(\Omega))^{4N},\;\;  \mathcal  H_1 = (H^1_0(\Omega))^{4N},\;\; \mathcal   H_{-1} = (H^{-1}(\Omega))^{4N}.$$
 Construct the enlarged  matrices $\widetilde {A}$ and $ \widetilde {D} $ by
\begin{equation} \label{coupling}\widetilde {A} = (a_{ij } I_{4}), \quad  \widetilde D = (d_{ij} I_4),
\end{equation}
where $A=(a_{ij })$ and $D=(d_{ij })$   are symmetric and positive semi-definite matrices of order $N$ and $I_4$ is the identity  of order $4$.

We consider the coupled system consisting of $N$ thermo-elastic subsystems \eqref{th-2}:
\begin{equation} \label{th-3}
U''  + \mathcal L U +  \widetilde A U +  \widetilde D  \mathcal G^* \mathcal G U'+   \widetilde D \mathcal G_0 U' = 0 \quad \hbox{in } \mathcal H_{-1},
\end{equation}
Due to  the specific form of the  matrices in \eqref{coupling}, we have
$$(\widetilde AU)^{(i)} = \sum_{j=1}^Na_{ij}u^{(j)},\quad  (\widetilde D \mathcal G^* \mathcal G U)^{(i)} = \sum_{j=1}^Nd_{ij}
 g^* gu^{(j)}, \quad (\widetilde D \mathcal G_0U)^{(i)} = \sum_{j=1}^Nd_{ij}g_0u^{(j)}
,\quad i=1, 2,\ldots, N.$$
In this context, the vector  $u^{(j)}$ is  treated as a single entity,  there is no interaction between  the components
$u^{(i)}_k$ and $u^{(j)}_l$ for $k\not =l$, that is, cross-component dependencies are not modeled.

On the other hand,  we have $$\|gu\|_{H_0} =  \|\theta\|_{H^1(\Omega)} \leqslant \|u\|_{H_1}\quad \hbox{and} \quad \|g_0u\|_{H_0}\leqslant \|u\|_{H_1}.$$
Therefore, the  operators $g$  and $g_0$ are continuous  from $(H^1_0(\Omega))^4$ to $(L^2(\Omega))^4$ and
satisfy  the regularity condition \eqref{1.5} with $r=1$. Moreover, it is easy to check that
$$\langle\! \langle  g_0u, u \rangle  \!\rangle_{H_0}= 0,\quad \forall u\in H_1.$$
So the coupling term $\widetilde DG_0U'$  in \eqref{th-3} has no impact on the system's dissipation.  Consequently, Theorem \ref{th1.1}  is still applicable to system \eqref{th-3}.
\medskip

\begin{theoreme}  \label   {th-thermo}    Assume $\Omega $ satisfies the   geometry control condition.  Assume furthermore that the matrices $A$ and $D $ satisfy Kalman's rank condition \eqref{1.4} and $\|AD-DA\|$ is sufficiently small.
 Then, there exists a constant $M>0$ such that   for any given  initial data $(U_0, U_1) \in \frak D(\mathcal A)$, the solution $U$ to  system  \eqref   {th-3} satisfies the polynomial decay estimate
\begin{equation} \|(U, U')\|_{ (H^1_0(\Omega))^{4N} \times (L^2(\Omega))^{4N}} \leqslant \frac{M}{t^{1/4}}\|(U_0, U_1)\|_{ (H^2(\Omega))^{4N}\times (H^1_0(\Omega))^{4N}},\;\; t\geqslant 1. \label{thermodecay}
\end{equation}
\end{theoreme}
\begin{proof}Clearly, the smallness of $\|AD-DA\|$ implies that of $\|\widetilde A\widetilde D-\widetilde D\widetilde A\|$.
 We next show that Kalman's rank condition \eqref{1.4} implies
$$\rank(\widetilde D, \widetilde A\widetilde D, \ldots \widetilde A^{4N-1}\widetilde D)=4N.$$
Otherwise, let $ \widetilde E \in \mathbb R^{4N}$ given by
$$\widetilde E= \begin{pmatrix} E^{(1)}\\\vdots\\E^{(N)}\end{pmatrix},\quad \hbox{where}\quad E^{(i)}\in \mathbb R^4,\quad i=1,\ldots, N$$ be an eigenvector of $\widetilde A$ contained in $\Ker(\widetilde D)$,  namely,
\begin{equation}\sum_{j=1}^Na_{ij}E^{(j)} =\lambda E^{(i)}\quad \hbox{and}\quad \sum_{j=1}^Nd_{ij}E^{(j)}=0, \quad i=1,\ldots, N.
\label{eigenvector}\end{equation}
Let
$E^{(k)}\not =0$ for some $1\leqslant k\leqslant N$.
By taking a scalar product on \eqref{eigenvector} with $E^{(k)}$, and setting
$$x_i=\langle E^{(k)}, E^{(i)}\rangle, \quad i=1,\ldots, N,$$
we get
$$\sum_{j=1}^Na_{ij}x_j =\lambda x_i\quad \hbox{and}\quad  \sum_{j=1}^Nd_{ij}x_j =0,\quad i=1,\ldots, N.$$
Since $x_k\not = 0$, then $x=(x_1,\dots, x_N)^N\in \mathbb R^N$  would be an  eigenvector of $A$ contained in $\Ker(D)$.
By Lemma \ref{th2.0} with $p=0$, this contradicts Kalman's rank condition \eqref{1.4}.

Then applying Theorem \ref{th1.1}  to system \eqref{th-3}, we obtain  the polynomial decay rate \eqref{thermodecay}.
Moreover, similarly to the decay rate \eqref{4.9} in Theorem \ref{th4.3}, the decay rate \eqref{thermodecay} is indeed sharp. \end{proof}

\bigskip

 \section{Optimality by spectral analysis}    \label{5}
 \subsection{Optimality}
We proceed to analyze the optimality of the results obtained in the previous section by means of spectral analysis.
In fact, Lemma \ref{th3.4} asserts that the decay rate ${1\over t^{\theta}}$ is optimal whenever the resolvent estimate \eqref{3.14} is established and condition \eqref{eq-opti} is fulfilled for at least one branch of the spectrum (\cite{liu2015}).

\medskip

\begin{example} \label{example1}
Let
\begin{equation}\label{5.0}
A= \begin{pmatrix}1& 0\\
0& 2 \end{pmatrix},\quad D= \begin{pmatrix}1& 2\\
2&4\end{pmatrix}, \quad U= \begin{pmatrix}u\\ v\end{pmatrix}. \end{equation}
We consider a special case of system \eqref{4.1} with these matrices.
\begin{equation}\label    {5.1} \begin{cases}u'' -\Delta u+u +u'+2 v'=  0 &{\rm in} \;\; \rr^+ \times \Omega, \\
v''-\Delta v+2v +2u' +4 v'=  0 &  {\rm in} \;\; \rr^+ \times \Omega,\\
u=v=0 & {\rm on} \;\; \rr^+ \times \Gamma.\end{cases}\end{equation}
\end{example}

It is clear that  assumption (i) of Theorem \ref{th1.1} is satisfied since
\begin{equation}\det(D)=0,\quad \rank(D, AD)=2,\quad \|AD-DA\|=2. \label    {5.2}  \end{equation}
Furthermore, the damping operator  $g=I$ satisfies the regularity condition \eqref{1.5} with $r=0$.
We now verify that the decay rate $\mathcal O(t^{-1/2})$ in \eqref{4.3} is in fact optimal.

For $\beta \in \mathbb R$, consider the corresponding eigenvalue problem:
 \begin{equation}  \label    {5.3} \begin{cases}\beta^2u -\Delta u+u +\beta u+2\beta  v=  0& \hbox{in } \Omega, \\
\beta^2v -\Delta v+2 v +2\beta u +4\beta  v=  0 & \hbox{in } \Omega,\\
u=v=0 & \hbox{on  } \Gamma.\end{cases}\end{equation}
Let  \begin{equation}  \begin{cases}-\Delta \phi =\nu_n^2\phi& \hbox{in } \Omega, \\
u=v=0 & \hbox{on  } \Gamma.\end{cases} \end{equation}
Setting
\begin{equation} \label{solution}
  u=a\phi, \quad v=b\phi
\end{equation}   in \eqref    {5.3},         we obtain
 \begin{equation}\begin{cases}(\beta^2 + \nu_n^2  + 1+\beta) a+  2b\beta =0,\\
2a\beta + (\beta^2 + \nu_n^2  +  2 + 4\beta) b=0,
\end{cases}\end{equation}
which admits a nontrivial solution if and only if
\begin{equation}(\beta^2 +\nu_n^2 +1+\beta)(\beta^2 +\nu_n^2 + 2 +4\beta) - 4\beta^2=0.\label    {5.4} \end{equation}
Let
$z=\beta^2 +\nu_n^2. $
Equation \eqref{5.4} becomes
 \begin{equation} z^2 +(3 +5\beta)z  + (2 +  6\beta) =0. \label    {5.5}  \end{equation}
A straightforward calculation yields
\begin{align*}2z&= -(3 +5\beta)\pm (25\beta^2+6\beta+ 1)^{1/2}\\&=- (3+5\beta)+5\beta
\Big(1+\frac{6}{25\beta} +\frac{1}{25\beta^2}\Big)^{1/2}\\
& = - \frac{12}{5} +\frac{8}{125\beta} + \frac{o(1)}{\beta^2}.\\
\end{align*}
It follows that
\begin{align} \label    {5.6} \beta_{n}&= i\nu_n\Big(1+\frac{6}{5\nu_n^2} -\frac{4}{125\nu_n^2\beta_n}+ \frac{o(1)}{\beta_n^2\nu_n^2}\Big)^{1/2}\\&
= i\nu_n\Big(1+\frac{3}{5\nu_n^2} - \frac{2 }{125\nu_n^2\beta_n} + \frac{o(1)}{\beta_n^2\nu_n^2}\Big) \notag \\
&
= i\nu_n+\frac{3i}{5\nu_n} - \frac{2i}{125\nu_n\beta_n} + \frac{o(1)}{\beta_n^2\nu_n}.  \notag\end{align}
Substituting $$ \beta_{n} = i\nu_n  +\frac{o(1)}{\nu_n}$$ into \eqref    {5.6},  we obtain the asymptotic expansion of a branch of eigenvalues:
\begin{align}\beta_{n}= i\Big(\nu_n +\frac{3}{5\nu_n}\Big)
 - \frac{2}{125\nu_n^2}+ \frac{o(1)}{\nu_n^{2}}. \end{align}
By Lemma \ref{th3.4}, the decay rate of system \eqref{5.1} cannot exceed $\mathcal O(t^{-1/2})$.
Thus, the decay rate $\mathcal O(t^{-1/2})$ in \eqref{4.3}, predicted by Theorem \ref{th4.1}, is indeed optimal.

This example confirms  also the optimality of the decay rate in \eqref{4.6}.

\medskip

\begin{example}   \label{example2}  Consider the following special case of system \eqref{4.7}:
\begin{equation}\begin{cases}u'' -\Delta u+u -\Delta u'-2 \Delta  v'=  0&{\rm in} \;\; \rr^+ \times \Omega, \\
v''-\Delta v+ 2v - 2\Delta u' -4 \Delta v'=  0  &{\rm in} \;\; \rr^+ \times \Omega,\\
u=v=0  &{\rm on} \;\; \rr^+ \times  \Gamma.\end{cases}\label{K-V}\end{equation}
\end{example}

It clear the coupling matrix $A$ and damping matrix $D$ are defined by \eqref{5.0} and satisfy assumption (i) of Theorem \ref{th1.1}. The damping operator $g=(-\Delta)^{1/2}$ satisfies the regularity condition \eqref{1.5} with $r=1$ by \eqref{gkv}.
We now verify the optimality of the decay rate $\mathcal O(t^{-1/4})$ in \eqref{4.9}.

For $\beta \in \mathbb R$, consider the eigenvalue problem
 \begin{equation}\label    {5.7} \begin{cases}\beta^2u -\Delta u+u -\beta \Delta u-2\beta \Delta  v=  0& \hbox{in } \Omega, \\
\beta^2v -\Delta v+ 2v -2\beta \Delta u -4\beta  \Delta v=  0 & \hbox{in } \Omega,\\
u=v=0 & \hbox{on  } \Gamma. \end{cases}\end{equation}
Let $u=a\phi,\; v=b\phi$ as in \eqref{solution}.
Then
$$\begin{cases}(\beta^2 + \nu_n^2  + 1+\beta\nu_n^2) a+  2b\beta\nu_n^2=0,\\
2a\beta\nu_n^2 + (\beta^2 + \nu_n^2  + 2+ 4\beta\nu_n^2) b=0,
\end{cases}$$
which has a nontrivial solution iff
\begin{equation}(\beta^2 +\nu_n^2 +1+\beta\nu_n^2)(\beta^2 +\nu_n^2 + 2+4\beta\nu_n^2) - 4\beta^2\nu_n^4=0.\label    {5.8} \end{equation}
Let
$$z=\beta^2 +\nu_n^2.$$
Equation \eqref{5.8} reduces to
$$z^2 +(3+5\beta\nu_n^2)z  + (2 +  6\beta\nu_n^2)=0.$$
Proceeding as in the analysis of \eqref{5.5}, we find
\begin{align*}2z&= -(3+5\beta\nu_n^2)\pm (25\beta^2\nu_n^4+6 \beta\nu_n^2+ 1)^{1/2}\\
&=-(3 +5\beta\nu_n^2)\pm 5\beta\nu_n^2
\Big(1+\frac{6}{25\beta\nu_n^2} +\frac{1}{25\beta^2\nu_n^4}\Big)^{1/2}\\
& = - \frac{12}{5} +\frac{8}{125\beta\nu_n^2} + \frac{o(1)}{\beta^2\nu_n^2}.\\
\end{align*}
A direct computation gives
\begin{align} \label    {5.9} \beta_{n}&= i\nu_n\Big(1+\frac{3}{5\nu_n^2} -\frac{4 }{125\nu_n^4\beta_n}+ \frac{o(1)}{\beta_n^2\nu_n^4}\Big)^{1/2}\\&
= i\nu_n\Big(1+\frac{3}{5\nu_n^2} - \frac{2}{125\beta_n\nu_n^4} + \frac{o(1)}{\beta_n^2\nu_n^4}\Big) \notag \\
&
= i\nu_n+\frac{3i}{5\nu_n} - \frac{2i}{125\beta_n\nu_n^{2}} + \frac{o(1)}{\beta_n^2\nu_n^{2}}.  \notag\end{align}
Substituting
$$ \beta_{n} = i\nu_n  +\frac{o(1)}{\nu_n}$$ into \eqref{5.9}, we obtain the asymptotic expansion of a branch of eigenvalues:
\begin{align}\beta_{n}= i\Big(\nu_n +\frac{3}{5\nu_n}\Big)
 - \frac{2}{125\nu_n^4}+ \frac{o(1)}{\nu_n^5}.\end{align}
By Lemma \ref{th3.4}, the decay rate of system \eqref{K-V} cannot exceed $\mathcal O(t^{-1/4})$.
Hence, the decay rate $\mathcal O(t^{-1/4})$ predicted by Theorem \ref{th4.3} is sharp.

\medskip

\begin{example}  \label{example3}
Let
 $$A= \begin{pmatrix}1& 0\\
0&0\end{pmatrix},\quad D= \begin{pmatrix}1& -1\\
-1&1\end{pmatrix},\quad U= \begin{pmatrix}u\\ v\end{pmatrix}.$$
 Consider the following special case of system \eqref{4.10}:
\begin{equation}\label    {5.10} \begin{cases}u'' - u_{xx} +u=  0 &{\rm in} \;\; \rr^+ \times (0,1), \\
v''- v_{xx} =  0 &{\rm in} \;\; \rr^+ \times (0,1),
\\ u(0)=v(0)=0 &{\rm in} \;\; \rr^+  ,
 \\u_x(1)  +  u'(1) - v'(1)=0  &{\rm in} \;\; \rr^+, \\
v_x(1) -u'(1) +  v'(1)=0 &{\rm in} \;\; \rr^+. \end{cases}\end{equation}
\end{example}
It is clear that
$$\det(D)=0,\quad \rank(D, AD)=2,\quad \|AD-DA\|=1.$$
Then   assumption (i) of Theorem \ref{th1.1} is true.
Moreover, by Theorem \ref{th4.5}, the Dirac mass $g=\delta_1$ satisfies the weaker regularity condition \eqref{4.20}.

We now verify the optimality of the decay rate $\mathcal O(t^{-1/2})$ in \eqref{4.17}.
For $\beta \in \mathbb R$, consider the eigenvalue problem
\begin{equation}\begin{cases}\beta^2u -u_{xx}+ u=  0,& 0<x<1,\\
\beta^2v -v_{xx}=  0, & 0<x<1,\\
u(0)=v(0)=0,\\ u_x(1) +\beta (u(1) - v(1))=0, \\
v_x(1) + \beta (-u(1) +  v(1))=0.\end{cases}\end{equation}
Let
$$\beta_1 =  \sqrt{\beta^2+1}, \quad u= a\sinh ( \beta_1 x),\quad v= b\sinh(\beta x).$$
The boundary conditions then yield
{  $$\begin{cases}a  (\beta_1 \cosh   \beta_1  +\beta \sinh   \beta_1 ) - b  \beta \sinh  \beta  = 0,\\
-a \beta  \sinh  \beta_1 + b \beta (\sinh \beta+  \cosh \beta)=0,\end{cases}$$
 }
 which admits a nontrivial solution iff
\begin{equation}\label    {5.12} \frac{ \beta_1 }{\beta}\cosh ( \beta_1 )(\sinh \beta+ \cosh \beta) +  \cosh  \beta  \sinh ( \beta_1 )=0.
\end{equation}
Using
$$\frac{ \beta_1 }{\beta} = 1+\frac{1 }{2\beta^2 } +\frac{o(1)}{\beta^4},\quad \beta_1 = \beta +\frac{1 }{2\beta } +\frac{o(1)}{\beta^{3}},$$
we expand equation \eqref{5.12} asymptotically in $1/\beta$.
This yields
$$\Big(2 +\frac{1 }{2\beta} + \frac{1 }{4\beta^2 } { + \frac{o(1)}{\beta^3} } \Big)\cosh\beta \sinh \beta +  \Big(\frac{1 }{2\beta}{ + \frac{o(1)}{\beta^3} }\Big)\sinh^2\beta + \Big(1+ \frac{1 }{2\beta} +\frac{5 }{8\beta^2 } { + \frac{o(1)}{\beta^3} } \Big)\cosh^2\beta=0.$$
Passing to $e^{\beta}$, we obtain
 \begin{align*}\Big(3 +\frac{3 }{2\beta}  + \frac{11}{8\beta^2} { + \frac{o(1)}{\beta^3} }\Big )e^{4\beta} +   \Big(2 + \frac{5}{4\beta^2 } { + \frac{o(1)}{\beta^3} } \Big)e^{2\beta} - \Big (1+\frac{1 }{8\beta^2} \Big) =\frac{o(1)}{\beta^3}.
 \end{align*}
It follows that
$$e^{2\beta} = -1+\frac{1 }{4\beta } - \frac{5 }{16\beta^2} +\frac{o(1)}{\beta^3}.$$
Then,
 \begin{align*}2\beta = \ln (-1)+  \ln \Big(1-\frac{1 }{4\beta } + \frac{5}{16\beta^2} + \frac{o(1)}{\beta^3} \Big)  =  \ln (-1) -{  \frac{1 }{4\beta } }+ \frac{9 }{32\beta^2}+\frac{o(1)}{\beta^3}.
  \end{align*}
From this we deduce the asymptotic expansion of a branch of eigenvalues:
 \begin{align}
 \beta_n=  \Big(n+{1\over2}\Big)\pi i  -{  \frac{1 }{8\beta_n }} + \frac{9 }{64\beta_n^2} + \frac{o(1)}{\beta_n^3}
 ={  \Big(  n\pi+{\pi\over2} + \frac{1}{8n\pi}\Big) i  -\frac{9}{64n^2  \pi^2 } } +\frac{o(1)}{n^3}. \notag\end{align}

Another branch of eigenvalues possesses a negative real part.
Therefore, the decay rate $\mathcal O(t^{-1/2})$ predicted by Theorem \ref{th4.5} is indeed optimal.

\bigskip
\subsection{Conclusion}
In this paper, we have developed a general framework for analyzing the polynomial stability of weakly coupled systems with a reduced number of feedback damping mechanisms. The abstract results have been applied to various models, including wave equations with local viscous or viscoelastic damping, and plate equations with local damping. In these cases, the derived decay rates are shown to be sharp. Furthermore, we have established the decay rates for wave equations with boundary damping. However, the optimality of the obtained decay rates remains an open question in the general case.

Despite these advances, a number of natural questions remain open for future research:
\begin{enumerate}

\item  What are the necessary and sufficient conditions for stability of the strongly coupled system $$U'' + \mathcal L U +A  \mathcal L^\beta U  +  D\mathcal G^*\mathcal GU'=  0,$$
where $\beta > 0,\; U, \; \mathcal L $ are defined in \eqref{z1.6} and $A,\; D$ denote the coupling and control matrices?

\item  Can one determine necessary and sufficient conditions for the stability of the more general system $$U'' + \mathcal L U +AU  + D \mathcal L^\gamma U' =  0,$$
where $ \gamma\ge0,\; U, \; \mathcal L $ are defined in \eqref{z1.6} and $A,\; D$ again represent coupling and control matrices?

\item  We aim to  consider the coupled system of wave equations with different speeds of propagation and develop a general theory
for the polynomial stability of second order evolutional equations.

\item {\Red How can the present analysis be extended to more general coupled systems, in which the coupling and control matrices are not necessarily symmetric or positive semidefinite, while still guaranteeing stability?}

\end{enumerate}

\end{document}